\documentclass{article}
\usepackage[T1]{fontenc}
\usepackage[latin9]{inputenc}
\usepackage{amsmath, amsthm, amssymb}
\usepackage{graphicx}
\usepackage{verbatim}
\usepackage{natbib}
\usepackage{algorithm}
\usepackage{algorithmic}
\usepackage{caption}
\usepackage{subcaption}
\usepackage{fancyvrb}
\usepackage{enumerate}
\usepackage{relsize}
\usepackage{algorithm}
\usepackage{hyperref}
\usepackage[margin=1.5in]{geometry}
\hypersetup{colorlinks,citecolor=blue,urlcolor=blue,linkcolor=blue}

\usepackage{stefan_tex}
\graphicspath{{./figures/}}
\usepackage{dsfont}

\theoremstyle{plain}
\newtheorem{prop}{Proposition}

\newtheorem{lem}[prop]{Lemma}
\newtheorem{thm}[prop]{Theorem}

\theoremstyle{definition}

\theoremstyle{remark}

\usepackage{authblk}

\usepackage{amsmath}               
  {
      \theoremstyle{plain}
      \newtheorem{assumption}{Assumption}
  }

\global\long\def\Jcal{\mathcal{J}}

\global\long\def\RR{\mathbb{R}}

\global\long\def\hbeta{\hat{\beta}}

\global\long\def\hgamma{\hat{\gamma}}

\global\long\def\htheta{\hat{\theta}}

\global\long\def\hmu{\hat{\mu}}

\global\long\def\htau{\hat{\tau}}

\global\long\def\oneb{\mathds{1}}

\global\long\def\PP{\mathbb{P}}

\global\long\def\EE{\mathbb{E}}

\author[$\star$]{Jelena Bradic}
\author[$\dagger$]{Stefan Wager}
\author[$\ddag$]{Yinchu Zhu}
\affil[$\star$]{University of California, San Diego}
\affil[$\dagger$]{Stanford University}
\affil[$\ddag$]{University of Oregon}
 
\date{\today}
\title{Sparsity Double Robust Inference of \\ Average Treatment Effects}

\begin{document}

\maketitle

\begin{abstract}
Many popular methods for building confidence intervals on causal effects under high-dimensional confounding
require strong ``ultra-sparsity'' assumptions that may be difficult to validate in practice. To alleviate this difficulty,
we here study a new method for average treatment effect estimation that yields asymptotically exact
confidence intervals assuming that either the conditional response surface or the conditional probability of
treatment allows for an ultra-sparse representation (but not necessarily both). This guarantee allows us to
provide valid inference for average treatment effect in high dimensions under considerably more generality than
available baselines. In addition, we showcase that our results are semi-parametrically efficient.
\end{abstract}

\section{Introduction}

Average treatment effect estimation is a core problem in causal inference, and has been the topic
of a considerable amount of recent literature \citep{imbens2015causal}. In this paper, we focus on the
task average treatment effect estimation with high-dimensional confounders:  We have access to $n$ \emph{i.i.d.}
samples $(X_i, \, Y_i, \, W_i) \in \xx \times \RR \times \cb{0, \, 1}$, where $X_i$ denotes high-dimensional
pre-treatment features ($\xx \subset \RR^p$ with $p \gg n$),
$W_i$ is the treatment assignment, and $Y_i$ is our outcome of interest. Causal effects are defined via potential
outcomes $\cb{Y_i(0), \, Y_i(1)}$, such that we observe $Y_i = Y_i(W_i)$ and the average treatment effect is
defined as $\tau = \EE{Y_i(1) - Y_i(0)}$ \citep{neyman1923applications,rubin1974estimating}.
Finally, we assume that there are no unmeasured confounders, i.e., the treatment assignment $W_i$ may not
be randomized, but can be treated as such once we control for $X_i$, i.e.,
$\cb{Y_i(0), \, Y_i(1)} \indep W_i \cond X_i$ \citep{rosenbaum1983central}.
Throughout, we also assume overlap, such that $\eta \leq \PP \p{W_i \cond X_i = x} \leq 1 - \eta$ for all $x$ and some
$\eta > 0$.

In the low-dimensional case, one of the most prominent approaches to average treatment effect estimation
is via augmented inverse-propensity weighting \citep{robins1994estimation},
\begin{equation}
\label{eq:AIPW}
\htau = \frac{1}{n} \sum_{i = 1}^n \p{\hmu_{(1)}\p{X_i} - \hmu_{(0)}\p{X_i} + \frac{W_i - \he(X_i)}{\he(X_i)(1 - \he(X_i)}\p{Y_i - \hmu_{(W_i)}(X_i)}},
\end{equation}
where $e(x) = \PP \p{W_i \cond X_i = x}$ is the propensity score, $\mu_{(w)}(x) = \EE \p{Y_i(w) \cond X_i = x}$ are
conditional response surfaces, and the quantities above with hats are estimates thereof.
A celebrated property of this estimator is that it is double robust, meaning that it is consistent whenever
either \smash{$\he(x)$} or the \smash{$\hmu_{(w)}(x)$} are consistent \citep{scharfstein1999adjusting}.
Moreover, \smash{$\htau$} is $\sqrt{n}$-consistent and semiparametrically efficient whenever the following
risk bounds hold \citep{farrell2015robust}
\begin{equation}
\label{eq:DR}
\EE{\p{\hmu_{(W)}(X) - \mu_{(W)}(X)}^2} \EE{\p{\he(X) - e(X)}^2} = o\p{\frac{1}{n}}.
\end{equation}
This statement is not sensitive to the structure of the estimators \smash{$\he(x)$} or the \smash{$\hmu_{(w)}(x)$}
provided we use an appropriate type of sample splitting \citep{chernozhukov2016double,zheng2011cross}, and thus
allows for considerable methodological flexibility. For example, \citet{farrell2018deep} establish conditions under
which \eqref{eq:DR} holds when \smash{$\he(x)$} or the \smash{$\hmu_{(w)}(x)$} are fit using neural networks.
These results on augmented inverse-propensity weighting can also be applied when $X_i$ is
high dimensional; however, in this case, the required risk bound can be difficult to satisfy.
In particular, except in extreme cases, the condition \eqref{eq:DR} effectively requires both
$\mu_{(w)}(x)$ and $e(x)$ to admit very sparse representations.

In this paper, we study a doubly robust construction that is specifically designed for
the high-dimensional case, and can be used for valid inference of $\tau$ under substantially
weaker sparsity assumptions than standard augmented inverse-propensity weighting.
We focus on the case where $\mu_{(w)}(x)$ and $e(x)$ have a high dimensional
linear-logistic specification (we omit intercepts for conciseness of presentation),
\begin{equation}
\label{eq:model}
\mu_{(w)}(x) = x' \beta_{(w)}, \ \ e(x) = 1/\p{1 + \exp(- x' \theta)}, \ \ \beta_{(w)}, \, \theta \in \RR^p,
\end{equation}
and consider an estimator that is $\sqrt{n}$-consistent for $\tau$ under the condition that
either $\theta$ or the $\beta_{(w)}$ (but not necessarily both) satisfy the type of sparsity
condition that is usually required for high-dimensional inference
\citep{javanmard2014confidence,van2014asymptotically,zhang2014confidence}.
We refer to this property as sparsity double robustness.

The issue of sparsity doubly robustness has been an open question since the recent development of high-dimensional inference. This literature requires sparsity level $o(\sqrt{n}/\log p)$ for inference, a condition stronger than $o(n/\log p)$ needed for consistent estimation. Such a gap has only been addressed very recently in \cite{javanmard2015biasing}, who found that the sparsity level of only one parameter needs to  satisfy  $o(\sqrt{n}/\log p)$, not both. However, their work only addresses the linear models and heavily relies on the Gaussianity assumption of the design. In this paper, we show that such sparsity doubly robustness result holds true for nonlinear models without Gaussian designs.

Our method starts with a functional form that closely resembles \eqref{eq:AIPW}. However, we choose
our estimators of $\mu_{(w)}(x)$ and $e(x)$ in  ways that carefully exploit the geometry of sparseness
in \eqref{eq:model} and are thus able to improve on its performance. A closely related estimator has been independently studied by \cite{tan2018model}, who considered potentially misspecified models but did not provide results on sparsity doubly robustness. Our main construction is as follows,
modulo some algorithmic tweaks (including a type of sample splitting):
\begin{align}
\label{eq:cbps}
&\htheta_{(w)} = \argmin_{\theta}\cb{\frac{1}{n} \sum_{i = 1}^n \p{ \oneb\{W_i \neq w\} X_i'\theta + \oneb \{W_i = w\} \exp\left(- X_i' \theta}\right) + \lambda_\theta \Norm{\theta}_1} \\
\label{eq:lasso}
&\hbeta_{(w)} = \argmin_{\beta} \cb{  \frac{1}{n} \sum_{W_i = w} \exp(- X_i' \htheta_{(w)})\p{Y_i - X_i' \beta}^2  + \lambda_\beta \Norm{\beta}_1} \\
\label{eq:tauhat}
&\htau = \frac{1}{n} \sum_{i = 1}^n \Big[\p{X_i'\hbeta_{(1)} + W_i  [1+\exp(-X_i'\htheta_{(1)})] (Y_i-X_i'\hbeta_{(1)})  }  \\
& \qquad\qquad\qquad -  \p{X_i'\hbeta_{(0)} + (1-W_i)  [1+\exp(-X_i'\htheta_{(0)})] (Y_i-X_i'\hbeta_{(0)})  } \Big].\nonumber
\end{align}
As discussed in Section \ref{sec:sdr}, we can study this estimator from two different perspectives.
If $\beta_{(w)}$ is very sparse, then the solution to \eqref{eq:lasso} converges at a fast rate, while the
solution to the propensity model \eqref{eq:cbps} effectively debiases \smash{$\hbeta_{(w)}$} even if
\smash{$\htheta_{(w)}$} is not particularly accurate. Meanwhile, if $\theta_{(w)}$ is very sparse, then the
converse holds. Our proof exploits this idea to establish sparsity double robustness.

The idea of  fitting a propensity model that can also leverage the shape of the conditional response surface
has generated considerable interest in recent years. The key observation here is that, in addition to being a
consistent estimator when $\theta$ is very sparse, \eqref{eq:cbps} also ``balances'' the inverse-propensity weighted features among
the treated and control samples in finite samples \citep{chan2015globally,hainmueller,imai2014covariate,tan2017regularized,zhao2016covariate}
\begin{equation}
\label{eq:balance}
\frac{1}{n} \sum_{i = 1}^n X_i \approx \frac{1}{n} \sum_{W_i = w} \frac{X_i}{1 + \exp(- X_i' \htheta_{(w)})}.
\end{equation}
The advantage of balancing is that, if the linear model for $Y$ is well specified, then balancing as in
\eqref{eq:balance} is sufficient for eliminating confounding, even when \smash{$\htheta_{(w)}$} itself
may be inconsistent or misspecified \citep{athey2016approximate,hirshberg2017balancing,kallus2018balanced,zhao2017entropy,zubizarreta2015stable}.
Note that, here, we estimate separate models for \smash{$\PP \p{W_i = 0 \cond X_i = x}$}
and \smash{$\PP \p{W_i = 1 \cond X_i = x}$}, parametrized by \smash{$\theta_{(0)}$} and \smash{$\theta_{(1)}$} respectively. This parametrization is based on (\ref{eq:model}) and reads
$$\PP \p{W_i=w\cond X_i=x}=1/(1+\exp(-x'\theta_{(w)})) \qquad \text{for}\qquad w\in\{0,1 \}. $$
Notice that by (\ref{eq:model}), we have that $\theta_{(1)}=\theta $ and $ \theta_{(0)}=-\theta $.
Asymptotically, we expect both parameter vectors to be consistent, \smash{$-\htheta_{(0)}, \, \htheta_{(1)} \approx \theta$},
but finite-sample differences between \smash{$\htheta_{(0)}$} and \smash{$\htheta_{(1)}$} play a key role in enabling the
balance \citep{imai2014covariate}.


Our main finding is that an estimator  constructed via the above ``balancing'' principle  achieves sparsity double robustness, meaning that it attains
$\sqrt{n}$-consistency given strong enough sparsity assumptions $o(\sqrt{n}/\log p)$ on either $\theta$ or the $\beta_{(w)}$, but not necessarily both. As discussed further below, this property is considerably stronger than the standard double robustness property \eqref{eq:DR} in the high-dimensional setup \eqref{eq:model}.


\subsection{Related Work}

Double robust and/or semiparametrically efficient estimation has a long tradition in the literature on causal inference
\citep{chernozhukov2016double,farrell2015robust,hahn1998role,hirano2003efficient,newey2018cross,
robins1,robins1994estimation,scharfstein1999adjusting,tan2010bounded,van2006targeted}.
More recently, it has been shown that with high dimensional confounders, we can improve the behavior
of double-robust-type estimators by having them directly exploit the geometry of sparsity.

As one of the first result in this direction, \citet{athey2016approximate} showed, 
given sufficient sparsity on the outcome function in \eqref{eq:model}, $\lVert \beta_{(w)} \rVert_0 \ll \sqrt{n} / \log(p)$,
we can achieve $\sqrt{n}$-consistency without any assumptions on the propensity score  beyond
overlap by simply using weights that balance moments as follows
(the \smash{$\hbeta_{(w)}$} are estimated via the lasso):
\begin{equation}
\label{eq:arb}
\begin{split}
&\htau = \frac{1}{n} \sum_{i = 1}^n X_i' \p{\hbeta_{(1)} - \hbeta_{(0)}} + \hgamma_i(W_i) (2W_i - 1) \p{Y_i - X_i'\hbeta_{(W_i)}}, \\
& \hgamma(w) = \argmin_\gamma \cb{\frac{1}{n^2} \sum_{W_i = w} \gamma_i^2 + \Norm{\frac{1}{n} \sum_{i = 1}^n \p{1 - \gamma_i \, \oneb\cb{W_i = w} } X_i}_\infty^2}.
\end{split}
\end{equation}
Conceptually, this approach is related to several papers that stress the important of covariate balance for accurate estimation
of treatment effects \citep{chan2015globally,imai2014covariate,kallus2018balanced,zhao2016covariate,zubizarreta2015stable}.
\citet{hirshberg2017balancing} establish conditions under which this estimator is efficient.

The main downside of the approximate residual balancing estimator \eqref{eq:arb} is that it always
requires sparsity of the outcome model, and cannot use a well specified and sparse propensity model
to compensate for a complex outcome model. Our sparsity double robustness result, which only requires
strong sparsity of either $\theta$ or the $\beta_{(w)}$ in \eqref{eq:model} directly addresses this limitation; and,
as shown in our experiments, yields substantial gains in accuracy when $\theta$ is in fact sparse.

Our result is most closely related to a recent proposal by \citet{chernozhukov2018double}, who studied any linear functional whose Riesz representer admits an (approximate) linear representation. In another paper, \citet{chernozhukov2016double} considers theoretical results for estimators based on 
learning conditional mean function and the propensity score. In both papers, the key condition is that the product of $\ell_2$-loss for learning the two nuisance parameters is $o(n^{-1/2})$, a condition referred to as rate double robustness; see Definition 2 in \cite{Smucler2019unifying}. Sufficient conditions for rate double robustness have been provided in these works in terms of sparsity levels. For example, Remark 5.2 of \citet{chernozhukov2016double} shows that rate double robustness is guaranteed when the product of two sparsity levels is $o(n)$, while Remark 7 of \citet{chernozhukov2018double} points out that under the assumption of bounded $\ell_1$-norm of both parameters, rate double robustness holds whenever one of the sparsity levels is $o(\sqrt{n}/\log p)$. 

The sparsity doubly robustness in this paper contributes to the literature by providing a different perspective. We show that efficient estimation is also possible in certain cases in which rate double robustness might not hold. One such example is  when the logistic parameter has bounded $\ell_1$-norm and has sparsity level $o(\sqrt{n}/\log p)$ and the conditional parameter has sparsity level $o(n^{3/4}/\log p)$ with potentially large $\ell_1$-norm. In this example,  we can still derive $1/\sqrt{n}$-consistency although we are not aware of any results that can guarantee rate double robustness. 

In addition, our work is also different from \citet{chernozhukov2018double} in terms of specification. In the context of average treatment effect estimation, the formulation in \citet{chernozhukov2018double} means
that we need there to exist (potentially sparse) vectors $\xi_{(0)}$ and $\xi_{(1)}$  whose $\ell_1$-norms are bounded (see Definition 3 or 4 therein) as well as such that $|1/(1 - e(x))  - x' \xi_{(0)}| \approx 0$ and $| 1/e(x) -x' \xi_{(1)}|\approx 0$ uniformly across $x$. This may be a reasonable assumption
if $x$ was in fact constructed as a basis expansion of some simpler measured features; however, it appears to be
difficult to justify more generally. One contribution of this paper relative to \citet{chernozhukov2018double} is that we achieve sparsity double robustness using the natural linear-logistic specification \eqref{eq:model}.

We also note two recent papers that consider estimators that resemble ours. \citet{ning2017high} consider an
estimator that, in the spirit of \citet{belloni2014inference}, first fit a penalized covariate-balancing propensity model,
and then re-fit without penalty those coefficients that correspond to features that are relevant to outcome modeling.
Meanwhile, \citet{tan2018model}  augments a penalized covariate-balancing propensity model in an outcome regression; it turns out that his covariate-balancing mechanism designed to address the issue of misspecification is also helpful for relaxing sparsity requirements.  Neither paper, however, achieves sparsity double robustness as discussed here; rather, they require both the outcome
parameter vector $\beta$ and the propensity parameter vector $\theta$ to be ultra-sparse---or, if there is misspecification
they require the population minimizers of both the outcome and propensity loss functions to be ultra-sparse.  Under the framework of \cite{Smucler2019unifying,Rotnizky2019mixed}, \citet{ning2017high,tan2018model} are classified as examples of model double robustness, which means that one of the models (either conditional mean or propensity score) is misspecified. Rate double robustness requires that the product of the $\ell_2$-norms of the estimation errors in two models is of the order $o(n^{-1/2})$. 

\section{Sparsity Double Robust Estimation}
\label{sec:sdr}

Whenever a parameter is identified through a moment condition, like \eqref{eq:AIPW}, a direct loss minimization that does not take into account this moment condition may not guarantee desirable properties.  Controlling inferential features of high-dimensional estimates is extremely difficult; most, if not all, require strict sparsity conditions. 
We aim to control optimality at estimation by directly embedding the leading term of the bias into a constraint of  newly designed estimators.
 
The main idea behind our construction is that we use estimators \smash{$\hbeta_{(0)}$}, etc.,
of $\beta_{(0)}$, etc., that have two complementary properties. When the underlying
parameter $\beta_{(0)}$ is ultra-sparse, then \smash{$\hbeta_{(0)}$} converges to $\beta_{(0)}$
in $\ell_1$-norm. Furthermore, even when $\beta_{(0)}$ is not ultra-sparse, \smash{$\hbeta_{(0)}$}
still has a useful covariate-balancing property implied by its Karush-Kuhn-Tucker (KKT) conditions
that can be put to good use (and similar guarantees hold for \smash{$\hbeta_{(1)}$}, \smash{$\htheta_{(0)}$}
and \smash{$\htheta_{(1)}$}).
We then provide two separate consistency and asymptotic normality proofs for our estimator: One that
assumes that $\beta_{(w)}$ is ultra-sparse and relies on  KKT conditions for the \smash{$\htheta_{(w)}$} estimator to
debias a very accurate \smash{$\hbeta_{(w)}$} estimator, and a second that assumes that $\theta$
is ultra-sparse and relies KKT conditions for the \smash{$\hbeta_{(w)}$} estimator to
debias a very accurate \smash{$\htheta_{(w)}$} estimator. Of course, only one of these
arguments needs to hold for us to achieve asymptotic normality, and thus our estimator is sparsity double robust.
This argument was inspired by the one used by \citet{chernozhukov2018double}; however, as discussed in the
related works section, \citet{chernozhukov2018double} make the somewhat unusual assumption that $1/e(x)$ can
 be approximated by a sparse linear model (rather than the assumption we make here, i.e., a sparse logistic model
 for $e(x)$).

 \subsection{Doubly Robust Balancing via Moment Targeting} \label{sec:kkt}

In this section, we briefly sketch the argument behind our main formal result, and use it to motivate
the form of our estimator. One of the main ingredients is moment targeting: We design estimators  such that they satisfy certain moment conditions that help reduce the bias at estimation. The  construction for estimators for $\beta_{(w)}$ and $\theta_{(w)} $ is based on the structure of the bias in the final estimator for $\EE \mu_{(w)}(X_i) $.  We emphasize that the argument here is only heuristic; formal arguments
are given in the appendix.

Given these preliminaries, observe that the treatment effect estimator under consideration can be written in a familiar form
\[
\hat \tau = \hat \mu_{(1)} - \hat \mu_{(0)},
\]
where $\hat\mu_{(w)}$ is an estimate of $\mu_{(w)} = \EE{Y_i(w)}$. We use
\[
\hat{\mu}_{(w)}=\frac{1}{n}\sum_{i=1}^{n}  X_{i}' \hat{\beta}_{(1)}  +  \hat \gamma_ i (w) \oneb\{W_i=w\} (Y_i - X_{i}' \hat{\beta}_{(w)}), \ \ \ \ \hat \gamma_i (w) =   1 + \exp(- X_i'\htheta_{(w)}),
\]
and construct \smash{$\hat{\mu}_{(0)}$} analogously. Our goal is to choose  $\htheta_{(w)}$ (and hence $\hgamma_{i}(w)$)
such as to control the errors \smash{$\hat{\mu}_{(w)} - \mu_{(w)}$} under flexible sparsity conditions.

To motivate our choice of $\htheta_{(1)}$, let us first consider the case where
$\beta_{(1)}$ is very sparse, i.e., $\|\beta_{(1)}\|_0\ll \sqrt{n}/\log p $. Notice that
\begin{align*}
\hat{\mu}_{(1)}-\mu_{(1)}  & =n^{-1}\sum_{i=1}^{n}(X_{i}'\beta_{(1)}-\mu_{(1)})+n^{-1}\sum_{i=1}^{n}W_{i}\varepsilon_{i,(1)}\hgamma_{i}(1)\\
&\qquad +n^{-1}\sum_{i=1}^{n}\left[1-W_{i}\hgamma_{i}(1)\right]X_{i}'\left(\hat{\beta}_{(1)}-\beta_{(1)}\right),
\end{align*}
where $\varepsilon_{i,(w)}=Y_i(w)-X_i'\beta_{(w)} $. The first two terms on the right hand side are asymptotically normal with mean zero under weak
consistency conditions on $\htheta_{(1)}$ that only require a moderate amount of sparsity on $\theta$.
Meanwhile, the last term can be bounded using Holder's inequality,
\begin{align*}
& \left| n^{-1}\sum_{i=1}^{n}\left[1-W_{i}\hgamma_{i}(1)\right]X_{i}'\left(\hat{\beta}_{(1)}-\beta_{(1)}\right) \right|\\
& =\left| n^{-1}\sum_{i=1}^{n}\left[1-W_{i}(1+\exp(-X_{i}'\hat{\theta}_{(1)}))\right]X_{i}'\left(\hat{\beta}_{(1)}-\beta_{(1)}\right) \right| \\
&\leq \left\Vert n^{-1}\sum_{i=1}^{n}\left[1-W_{i}(1+\exp(-X_{i}'\hat{\theta}_{(1)}))\right]X_{i} \right\Vert_{\infty} \left\Vert \hat{\beta}_{(1)}-\beta_{(1)}\right\Vert_{1}.
\end{align*}
Under sparsity assumption $\|\beta_{(1)}\|_{0}=o(\sqrt{n}/\log p)$, we can typically obtain $\|\hat{\beta}_{(1)}-\beta_{(1)}\|_{1}=o_P(1/\sqrt{\log p})$ via sparse methods \citep[e.g.,][]{negahban2012unified}.
Meanwhile, the KKT conditions for the estimator in \eqref{eq:cbps} with $w=1$ automatically yields \citep{tan2017regularized}
$$
\left\Vert n^{-1}\sum_{i=1}^{n}\left[1-W_{i}(1+\exp(-X_{i}'\hat{\theta}_{(1)}))\right]X_{i} \right\Vert_{\infty}=O_P(\sqrt{n^{-1}\log p}),
$$
thus bounding the bias to the order of $o_P(n^{-1/2})$. This is the first example of moment targeting. The KKT condition of the estimator provides a convenient moment condition for the purpose of bias reduction. 

The above argument closely mirrors the argument used by \citet{athey2016approximate} to obtain
$\sqrt{n}$-consistent estimates of $\tau$ when $\|\beta_{(w)}\|_0\ll \sqrt{n}/\log p$. The main difference
  with our approach is that \citet{athey2016approximate} do not fit a model for $\theta$, but instead
directly optimize the weights \smash{$\hgamma$} via quadratic programming as in
\citet{javanmard2014confidence} and \citet{zubizarreta2015stable}. That in turn, leads to somewhat loss of flexibility whenever the outcome model is not sparse.

Here, the fact that we also model $\theta$
enables us to alternatively exploit sparsity in $\theta$ and correspondingly relax assumptions on $\beta_{(1)}$.
To do so, note that
\begin{align*}
\hat{\mu}_{(1)}-\mu_{(1)}  & =n^{-1}\sum_{i=1}^{n}(X_{i}'\beta_{(1)}-\mu_{(1)})+n^{-1}\sum_{i=1}^{n}W_{i}\varepsilon_{i,(1)}\hgamma_{i}(1)\\
&\qquad +n^{-1}\sum_{i=1}^{n}\left[1-W_{i}(1+\exp(-X_{i}'\theta))\right]X_{i}'\left(\hat{\beta}_{(1)}-\beta_{(1)}\right)\\
&\qquad +n^{-1}\sum_{i=1}^{n}W_{i}\left[\exp(-X_{i}'\theta)-\exp(-X_{i}'\hat{\theta}_{(1)})\right]X_{i}'\left(\hat{\beta}_{(1)}-\beta_{(1)}\right).
\end{align*}
Again, the sum of the first three terms above is asymptotically Gaussian on the $\sqrt{n}$-scale
under only weak assumptions on $\hbeta_{(1)}$.
To handle the last term, we can use Taylor expansion to argue that (we will make this rigorous in the proof of our main result)
\begin{multline*}
\left| n^{-1}\sum_{i=1}^{n}W_{i}\exp(-X_{i}'\hat{\theta}_{(1)})X_{i}'(\hat{\theta}_{(1)}-\theta)X_{i}'\left(\hat{\beta}_{(1)}-\beta_{(1)}\right)\right|\\
\lesssim \left\Vert n^{-1}\sum_{i=1}^{n}W_{i}\exp(-X_{i}'\hat{\theta}_{(1)})X_{i}X_{i}'\left(\hat{\beta}_{(1)}-\beta_{(1)}\right)\right\Vert_{\infty} \left\Vert \hat{\theta}_{(1)}-\theta \right\Vert_{1}.
\end{multline*}
Now, given sufficient sparsity on $\theta$, i.e., $\|\theta\|_0\ll \sqrt{n}/\log p $ we can verify that
$\|\hat{\theta}_{(1)}-\theta \|_1=o_P(1/\sqrt{\log p})$. Meanwhile, the
the KKT condition for the estimator $\hat{\beta}_{(1)}$ in \eqref{eq:lasso} automatically yields that
the first component above is $O_P(\sqrt{n^{-1}\log p})$.
Thus, we also expect $\hmu_{(0)}$ to be accurate when $\theta$ is very sparse, even if $\beta$ is not. This is another example of moment targeting in that the KKT condition for $\hbeta_{(1)} $ again provides a convenient bound for bounding the bias.

\subsection{Sample splitting for optimality}

The above discussion provides some helpful conceptual guidance on how to pick good estimators of the unknown $\beta_{(w)}$ and $\theta_{(w)}$.
To achieve optimality in most general terms, we invoke a special scheme of sample-splitting similar to cross-fitting. Under the usual cross-fitting scheme, the influence function is evaluated on observations that are not used to estimate the nuisance parameters (in our case $\beta_{(w)}, \theta_{(w)} $). Cross-fitting has been used to reduce bias terms in many semiparametric and high-dimensional models \citep[see, e.g.,][]{chernozhukov2016double,newey2018cross,schick1986asymptotically,zheng2011cross}. Here, our approach requires us to only cross-fit $\hat \beta_{(w)}$, but not $\hat \theta_{(w)}$.


The entire sample is divided into two parts $\Jcal$ and $\Jcal^c$.   For $F\in\{\Jcal,\Jcal^c,\}$, estimators trained using the sample $F$ are denoted with $\hat \beta_{(w),F}$ and $\hat \theta_{(w),F}$, respectively. For expositional simplicity, we assume that $|\Jcal|=|\Jcal^c|=n/2 $. 
Then, for $ (w,F)\in \{0,1\}\times \{\Jcal,\Jcal^c \}$, we define the estimator of the mean
\begin{equation} \label{eq: mu_w_J}
\hat{\mu}_{(w),F }=\frac{1}{ |F|} \sum_{i \in F}  \p{ X_{i}' \hat{\beta}_{(w),F^c}    +  \hat \gamma_ i (w, F)  \oneb \{W_i=w\} (Y_i - X_{i}' \hat{\beta}_{(w),F^c}) } 
\end{equation}
where 
 the weight function is defined in-sample
\[
\hat \gamma_i (w, F  ) =  1 +  \exp(- X_i'\htheta_{(w),F})  ,
\]
  $\hat \theta_{(w), F}$ is defined in Algorithm \ref{alg1}, and  $\hat \beta_{(w),F}$ is given by
\begin{equation}\label{eq: est beta}
\hbeta_{(w),F} = \arg\min_{\beta} \cb{  \frac{1}{|F| } \sum_{ i \in F} \oneb\{W_i = w\} \exp(-  X_i' \htheta_{(w),F})\p{Y_i - X_i'\beta}^2  + \lambda_\beta \Norm{\beta}_1}.
\end{equation}
Algorithm \ref{alg1} presents details of the propensity estimation.
The   loss functions in \eqref{eq: est beta} and \eqref{eq:loss2} were recently utilized in \citet{tan2018model} but the  proposed average treatment effects estimator therein does not achieve sparsity double robustness.

 


  \begin{algorithm} [H]
\caption{Optimistic penalized covariate-balancing propensity estimation} 
\label{alg1} 
\begin{algorithmic} 
    \REQUIRE - a training sample $F\in\{\Jcal,\Jcal^c \} $, a treatment status indicator $w \in \{0,1\}$ a tuning parameter  $\lambda_\theta \asymp \sqrt{\log(p)/n}$ and a pre-defined constant $\kappa$
    \STATE  Compute
    \begin{equation}\label{eq:loss2}
\check \theta_{(w), F}  \leftarrow  \arg\min_{\theta}\cb{\frac{1}{|F|} \sum_{i \in F}\left[ \oneb\{W_i\neq w\}X_i' \theta +\oneb\{W_i=w\}\exp\p{  -X_i'\theta  } \right] + \lambda_\theta \Norm{\theta}_1 }
    \end{equation}
    \IF{$ \| \check \theta_{(w), F} \|_1 > \kappa$,}
        \STATE  
         \[
\hat  \theta_{(w), F}  \leftarrow  \arg\min_{\theta}\cb{ \| \theta\|_1, \mbox{ s.t. }  \left \|  \frac{1}{|F|} \sum_{i \in F} \Big[  1- \oneb\{W_i=w\} \bigl(  1+\exp(-X_i'  \theta)\bigl)\Big] X_{i}  \right \|_{\infty}  \leq \lambda_\theta } 
    \]
    \ELSE
        \STATE $\hat  \theta_{(w), F} \leftarrow \check \theta_{(w), F}$
          \ENDIF
           \RETURN $\hat  \theta_{(w), F} $
\end{algorithmic}
\end{algorithm}

The method presented here splits the sample into two subsamples. Although one can easily follow the same principle and split the sample  into multiple subsamples, we  do  not pursue this option here  for notational simplicity. We now define 
%
%
\[
\hat \mu _{(1)} =(\hat{\mu}_{(1),\Jcal }+\hat{\mu}_{(1),\Jcal^c })/2.
\]
Similarly, we can define $\hat \mu _{(0)} $.
Then, the average treatment effect estimator is defined as 
\begin{equation}\label{eq:tau}
\hat \tau = \hat \mu _{(1)}-\hat \mu _{(0)}.
\end{equation}

\section{Formal Results}

We now turn to a formal characterization of the average treatment effect estimator in (\ref{eq:tau}), with the aim of providing asymptotic Gaussianity whenever one, but not both, of $\theta$, $\beta$ is estimated consistently. We begin by listing some theoretical assumptions necessary for the development of the theoretical guarantees. 

First, we assume that the covariate space and the parameter space are both
subsets of Euclidean space; specifically, we assume that $X \in [a,b]^p$ and $\theta \in \mathcal{B}_1 (r) \subset \mathbb{R}^p$ for some bounded $r>0$,
where $\mathcal{B}_1(r)$ is an $\ell_1$ ball with radius $r$.  
 
 The results discussed below hold whenever, 
 the tuning parameters (in Algorithm \ref{alg1}) are chosen to be proportional to $\sqrt{\log(p)/n}$.
  Moreover, we also assume that $\kappa$ is chosen to be larger than $\| \theta_{(w)}\|_1$. Our procedure is not particularly sensitive to the choice of $\kappa$;   in practice it suffices to choose  a large enough number.

 \begin{assumption}[Eigenvalue]\label{assu: RE condition}
 The minimum and maximum eigenvalues of  $\mathbb{E}[ X_iX_i']$ are contained in a bounded interval that does not contain zero. 
 \end{assumption}

 Our next assumption controls the regularity properties of the errors within both models \eqref{eq:model}.
 Let $\varepsilon_{i,(w)} = Y_i - X_i'\beta_{(w)}$ and $v_{i,(w)} = \oneb\{W_i=w\} - e_{(w)}(X_i)$. Note that in the context of models \eqref{eq:model} the unconfoundedness assumption implies  $\varepsilon_{i,(w)} \perp v_{i,(w)} | X_i$ and from now on we will work with this slightly weaker assumption. 
 
  \begin{assumption}[Model]\label{assu: unconfoundedness}
  $X_i $ has a bounded sub-Gaussian norm. Moreover, for $w\in\{0,1\}$, $ \varepsilon_{i,(w)}$ is sub-Gaussian. 
 \end{assumption}

 Now, observe that Assumption \ref{assu: unconfoundedness} is  very weak and in particular it is not implying consistent estimation in  the outcome  model. The boundedness of $\|X_i\|_\infty$ and $\|\theta\|_1$ guarantees the overlap condition.

 Finally, in the context of the average treatment effects, in order to provide confidence intervals an estimate of the   asymptotic variance of $\hat \tau$  is needed. We show   an asymptotic variance  of $\hat \tau$ takes the form of 
 \begin{align*}
  & \mathbb{E} \left[	X_{i}'  \left(\beta_{(1)}  - \beta_{(0)} \right)    - \tau \right]^2  +\mathbb{E}\left[ W_{i}\varepsilon_{i,(1)}  \gamma_i (1) \right]^2 
+\mathbb{E}\left[ (1-W_{i})\varepsilon_{i,(0)}  \gamma_i (0) \right]^{2} 
\\
& \qquad : = \Omega + V_{(1)} + V_{(0)}.
\end{align*}
Observe that $\Omega$ is the variability  induced primarily from the variability of the design $X$. The other two terms  can be viewed as properly normalized unexplained variance of the models \eqref{eq:model}.

To define variance estimates, we define estimates of $\Omega$, $V_{(1)}$ and  $V_{(0)}$ separately.
We set 
\begin{equation}\label{eq:Omega}
\hat \Omega = n^{-1} \sum_{i=1}^n \left( X_{i}'   (\hat \beta_{(1)}  - \hat \beta_{(0)}  ) - \hat \tau\right)^2,
\end{equation}
as well as 
\begin{equation}\label{eq:V_w}
\begin{split}
&\hat V_{(w)} =n^{-1} \sum_{i=1}^n   (2W_i -1)^2 \hat \varepsilon_{i,(w)}^2  \hat \gamma_i^2 (w)  \ind\{W_i =w\}, \\
&\hat \varepsilon_{i,(w)} = Y_i - X_i' \hat \beta_{(w)}, \qquad \hat \gamma_i(w) = 1+e^{- (2W_i -1) X_i' \hat \theta_{(w)}}.
\end{split}
\end{equation}
In the above display, $\hat \beta_{(w)}$ and $\hat \theta_{(w)}$ could be the ones computed on one sample, $\mathcal{J}$ or could be different ones; for example to borrow strength across samples we consider 
\[
 \hat \beta_{(w)}   = (\hat \beta_{(w),\mathcal{J}}  +\hat \beta_{(1),\mathcal{J}^c} )/2, \qquad   \hat \theta_{(w)} = (\hat \theta_{(w),\mathcal{J}} +\hat \theta_{(w),\mathcal{J}^c})/2 .
\]
Now, we define the variance estimate as 
\begin{equation}\label{eq:hatV}
\hat V = \hat \Omega+\hat V_{(0)}+\hat V_{(1)}
\end{equation}
for $\hat \Omega$ and $\hat V_{(w)}$ defined in \eqref{eq:Omega} and \eqref{eq:V_w}, respectively. 
We show that the construction above is appropriate for such circumstances and leads to asymptotically optimal confidence sets.

 \begin{thm}\label{thm: main}
 Let Assumptions \ref{assu: RE condition} and \ref{assu: unconfoundedness} hold. Then, 
 as long as one of the following two conditions holds,
\begin{itemize}
\item[\mbox{\rm (i)}] {\mbox{\rm (Ultra-sparse outcome model)}} $\|\beta_{(w)}\|_{0}=o(\sqrt{n}/\log p)$
and $\|\theta\|_{0}=o(n/\log p)$,
\item[\mbox{\rm (ii)}] {\mbox{\rm (Ultra-sparse propensity model)}} $\|\theta\|_{0}=o(\sqrt{n}/\log p)$
and $\|\beta_{(w)}\|_{0}=O(n^{3/4}/\log p)$,
\end{itemize}
 we have the following representation of $\hat \tau$ as defined in \eqref{eq:tau}
 \[
\sqrt{n}(\hat{\tau}-\tau)=n^{-1/2}\sum_{i=1}^{n} \psi(Y_{i},W_{i},X_{i},\tau,e_{(1)}(X_{i})) +o_{P}(1),
\]
where 
\begin{equation} \label{eq: influence fun}
\begin{split}
&\psi(Y_{i},W_{i},X_{i},\tau,e_{(1)}(X_{i}))= \\
&\ \ \ \ \ \ X_{i}'\p{\beta_{(1)} - \beta_{(0)}} + W_i \p{\frac{Y_{i} - X_{i}'\beta_{(1)}}{e_{(1)}(X_{i})}} - (1-W_{i})\p{\frac{Y_{i} - X_{i}'\beta_{(0)}}{1-e_{(1)}(X_{i})}}-\tau.
\end{split}
\end{equation}
In particular, under these assumptions we have 
 \[
\sqrt{n}(\hat{\tau}-\tau)  \overset{\text{d}}{\rightarrow}  \mathcal{N}(0,V_*), \ \ \ \ \ 
V_* = \mathbb{E} \left[ \psi^2(Y_{i},W_{i},X_{i},\tau,e_{(1)}(X_{i}))\right].
 \]
Moreover,  under the same assumptions, for $\hat V$ defined in \eqref{eq:hatV} we have
\[
\hat{V}=V_*+o_{P}(1),
\]
and in turn $\sqrt{n} (\hat \tau - \tau) /\sqrt{\hat V}   \overset{\text{d}}{\rightarrow}  \mathcal{N}(0,1 )$.
 \end{thm}

By well known results \citep[e.g.,][]{hahn1998role,newey1994asymptotic,robins1}, $\psi$ in (\ref{eq: influence fun}) is the efficient influence function and $V_*$  is the semiparametric efficiency lower bound. Therefore, our estimator $\htau$ in (\ref{eq:tau}) is a semiparametrically efficient estimator. As discussed before, the sparsity requirement in Theorem \ref{thm: main} is considerably weaker that needed by existing estimators in the high-dimensional linear-logistic model \eqref{eq:model}, including the methods discussed in \citet{athey2016approximate},
\citet{belloni2014inference}, \citet{farrell2015robust}, \citet{ning2017high} and \citet{tan2018model},
in that we only need either the outcome model or the propensity model to be ultra sparse (but not both).


\section{Numerical Experiments}

In this section we present numerical work where we contrast the behavior of the introduced method with the existing approaches. 
We consider the following design setting,   $X_{i}\sim N(0,\Sigma)$ with $\Sigma_{i,j}=\rho^{|i-j|}$
and $\rho=0.6$. We set the sample size  and the number of covariates to be $n=500$ and   $p=600$, respectively. The following structure for the parameters is used. 
We consider the propensity model where
$$\theta=a_{\theta}(1,0,1,0,1,0,...,0,1,0,0...,0)'$$
with $\|\theta\|_{0}=s_{\theta}$.  We vary the value of $s_{\theta}$ and  set $a_{\theta}$ such that $\sqrt{\theta'\Sigma_{X}\theta}=1$.

Similarly,  for  the outcome model we consider 
$$\beta_{(1)}=a_{\beta}(1,0,1,0,1,0,...,0,1,0,0...,0)'$$
and set $\beta_{(0)}=-\beta_{(1)}$. In other words, non-zero entries
appear only on indices with odd numbers.  For $a_\beta$, we consider two cases.
In the first, we have homoskedastic Erros: The error term is generated from a centered $\chi^{2}(1)$ distribution (since it is light tailed and asymmetric). We set  $\sqrt{\beta_{(1)}'\Sigma_{X}\beta_{(1)}}=\sqrt{2}$. We use $\sqrt{2}$ to get an $R^{2}$ of	50\%  (because the error has variance 2).
The second has heteroskedastic errors: $\varepsilon_{i,(1)} $ is generated according to 
	$$\p{4\times\mathbf{1}\{e(X_{i})\leq0.5\}+\mathbf{1}\{e(X_{i})>0.5\}}\xi_i$$ with $\xi_i$ being a centered $\chi^2(1)$ variable independent of $X_i$. Observe that $\varepsilon_{i,(0)}$ is still a centered  $\chi^2(1)$ variable.
 We also consider other values of $a_\beta$ such that the $R^2$ in the homoskedastic case is 10\%. We report the mean squared error (MSE) and coverage probability of 95\% confidence
interval (CP). We compare our methods with two popular alternatives:
\begin{itemize}
	\item \textbf{AIPW.} Augmented inverse propensity weighting \citep{robins1994estimation} is a popular method for estimating the average treatment effect. We implement this using the hdm package in R.
	\item \textbf{ARB.} Approximate residual balancing was proposed by \cite{athey2016approximate}. This method can handle cases in which the propensity score is hard to estimate. 
\end{itemize}
 
\begin{table}[H]
	\caption{Mean squared error (MSE) and coverage probability (CP) across two models both with R-squared= $0.5$. Comparison includes augmented inverse propensity weighting (AIPW), approximate residual balancing (ARB), and  sparsity double robust estimation (SDR) in (\ref{eq:tau}). Parameters $s_\theta$ and $s_\beta$ denote sparsity of the propensity and outcome models, respectively.}
	\label{tab: 1}
\begin{center}
		\begin{tabular}{cr@{\extracolsep{0pt}.}lr@{\extracolsep{0pt}.}lcr@{\extracolsep{0pt}.}lr@{\extracolsep{0pt}.}l}
		& \multicolumn{2}{c}{} & \multicolumn{2}{c}{} &  & \multicolumn{2}{c}{} & \multicolumn{2}{c}{}\tabularnewline
		& \multicolumn{9}{c}{Homoscedastic errors}\tabularnewline
		  \cline{4-9} \\
		& \multicolumn{4}{c}{$s_{\theta}=2$, $s_{\beta}=2$} &  & \multicolumn{4}{c}{$s_{\theta}=2$, $s_{\beta}=30$}\tabularnewline
		& \multicolumn{2}{c}{MSE} & \multicolumn{2}{c}{CP} &  & \multicolumn{2}{c}{MSE} & \multicolumn{2}{c}{CP}\tabularnewline
		AIPW & 0&041 & 0&954 &  & 0&063 & 0&950\tabularnewline
		ARB & 0&039 & 0&928 &  & 0&043 & 0&916\tabularnewline
		SDR & 0&042 & 0&952 &  & 0&037 & 0&960\tabularnewline 
		& \multicolumn{4}{c}{$s_{\theta}=30$, $s_{\beta}=2$} &  & \multicolumn{4}{c}{$s_{\theta}=30$, $s_{\beta}=30$}\tabularnewline
		& \multicolumn{2}{c}{MSE} & \multicolumn{2}{c}{CP} &  & \multicolumn{2}{c}{MSE} & \multicolumn{2}{c}{CP}\tabularnewline
		AIPW & 0&058 & 0&932 &  & 0&097 & 0&954\tabularnewline
		ARB & 0&039 & 0&882 &  & 0&043 & 0&924\tabularnewline
		SDR & 0&035 & 0&972 &  & 0&038 & 0&968\tabularnewline 
				& \multicolumn{2}{c}{} & \multicolumn{2}{c}{} &  & \multicolumn{2}{c}{} & \multicolumn{2}{c}{}\tabularnewline
		& \multicolumn{9}{c}{Heteroscedastic errors}\tabularnewline
			  \cline{4-9} \\
		& \multicolumn{4}{c}{$s_{\theta}=2$, $s_{\beta}=2$} &  & \multicolumn{4}{c}{$s_{\theta}=2$, $s_{\beta}=30$}\tabularnewline
		& \multicolumn{2}{c}{MSE} & \multicolumn{2}{c}{CP} &  & \multicolumn{2}{c}{MSE} & \multicolumn{2}{c}{CP}\tabularnewline
		AIPW & 0&223 & 0&920 &  & 0&191 & 0&958\tabularnewline
		ARB & 0&145 & 0&922 &  & 0&129 & 0&950\tabularnewline
		SDR & 0&132 & 0&972 &  & 0&081 & 0&990\tabularnewline 
		& \multicolumn{4}{c}{$s_{\theta}=30$, $s_{\beta}=2$} &  & \multicolumn{4}{c}{$s_{\theta}=30$, $s_{\beta}=30$}\tabularnewline
		& \multicolumn{2}{c}{MSE} & \multicolumn{2}{c}{CP} &  & \multicolumn{2}{c}{MSE} & \multicolumn{2}{c}{CP}\tabularnewline
		AIPW & 0&245 & 0&934 &  & 0&258 & 0&962\tabularnewline
		ARB & 0&113 & 0&948 &  & 0&102 & 0&946\tabularnewline
		SDR & 0&080 & 0&992 &  & 0&074 & 0&994\tabularnewline
		& \multicolumn{2}{c}{} & \multicolumn{2}{c}{} &  & \multicolumn{2}{c}{} & \multicolumn{2}{c}{}\tabularnewline
	\end{tabular}
\end{center}

\end{table}

The results are reported in  Tables \ref{tab: 1} and \ref{tab: 2}. Table \ref{tab: 1} indicated that in the baseline case with extremely sparse $\beta_{(w)}$ and $\theta$, AIPW, approximate residual balancing method and SDR perform very similarly. Note that this is as expected;  all of the methods should be achieving the same asymptotic variance. However, when either the propensity score model or the conditional mean function or both are not extremely sparse, the SDR method delivers smaller MSE. This confirms our theoretical results, which state that our method is guaranteed to provide efficient estimation even if there is lack of extreme sparsity. 

Perhaps a more direct way of formalizing this intuition is through analyzing the rate of the remainder similar to the discussion in \citep{newey2018cross}; one way is to express the rate of the remainder for the asymptotic expansion in Theorem \ref{thm: main} in terms of $\|\beta_{(w)}\|_0$ and $\|\theta\|_0$. One can use our technical arguments to show that, compared to AIPW, the remainder of the SDR estimator has the same order or smaller order of magnitude. In Table \ref{tab: 2}, we also report the results by setting $a_\beta$ such that in the homoscedasticity case we have an R-squared of 10\%. The pattern is quite similar; we observe comparable performance when we have extreme sparsity in both models and the proposed method has lower MSE in the absence of such sparsity in either model.

\begin{table}[H]
	\caption{Mean squared error (MSE) and coverage probability (CP) across two models both with R-squared= $0.1$. Comparison includes augmented inverse propensity weighting (AIPW), approximate residual balancing (ARB), and  sparsity double robust estimation (SDR) in (\ref{eq:tau}). Parameters $s_\theta$ and $s_\beta$ denote sparsity of the propensity and outcome models, respectively.}
	\label{tab: 2}
	\begin{center}
		\begin{tabular}{cr@{\extracolsep{-0.2pt}.}lr@{\extracolsep{-0.2pt}.}lcr@{\extracolsep{-0.20pt}.}cr@{\extracolsep{-0.20pt}.}l}
			& \multicolumn{2}{c}{} & \multicolumn{2}{c}{} &  & \multicolumn{2}{c}{} & \multicolumn{2}{c}{}\tabularnewline
			& \multicolumn{9}{c}{Homoscedastic errors}\tabularnewline
				  \cline{4-9} \\
			& \multicolumn{4}{c}{$s_{\theta}=2$, $s_{\beta}=2$} &  & \multicolumn{4}{c}{$s_{\theta}=2$, $s_{\beta}=30$}\\
			& \multicolumn{2}{c}{MSE} & \multicolumn{2}{c}{CP} &  & \multicolumn{2}{c}{MSE} & \multicolumn{2}{c}{CP}\\
			AIPW & 0&043 & 0&954 &  & 0&055 & 0&952\tabularnewline
			ARB & 0&038 & 0&915 &  & 0&039 & 0&926\tabularnewline
			SDR & 0&042 & 0&948 &  & 0&035 & 0&965\tabularnewline
			& \multicolumn{4}{c}{$s_{\theta}=30$, $s_{\beta}=2$} &  & \multicolumn{4}{c}{$s_{\theta}=30$, $s_{\beta}=30$}
			\tabularnewline
			& \multicolumn{2}{c}{MSE} & \multicolumn{2}{c}{CP} &  & \multicolumn{2}{c}{MSE} & \multicolumn{2}{c}{CP}\tabularnewline
			AIPW & 0&048 & 0&947 &  & 0&086 & 0&962\tabularnewline
			ARB & 0&039 & 0&900 &  & 0&042 & 0&920\tabularnewline
			SDR & 0&035 & 0&977 &  & 0&039 & 0&974 
			\\
					& \multicolumn{2}{c}{} & \multicolumn{2}{c}{} &  & \multicolumn{2}{c}{} & \multicolumn{2}{c}{}\tabularnewline
	& \multicolumn{9}{c}{Heteroscedastic errors}\tabularnewline
		  \cline{4-9} \\
	& \multicolumn{4}{c}{$s_{\theta}=2$, $s_{\beta}=2$} &  & \multicolumn{4}{c}{$s_{\theta}=2$, $s_{\beta}=30$}\tabularnewline
	& \multicolumn{2}{c}{MSE} & \multicolumn{2}{c}{CP} &  & \multicolumn{2}{c}{MSE} & \multicolumn{2}{c}{CP}\tabularnewline
	AIPW & 0&206 & 0&936 &  & 0&211 & 0&932\tabularnewline
	ARB & 0&116 & 0&950 &  & 0&132 & 0&946\tabularnewline
	SDR & 0&073 & 0&980 &  & 0&079 & 0&990\tabularnewline 
	& \multicolumn{4}{c}{$s_{\theta}=30$, $s_{\beta}=2$} &  & \multicolumn{4}{c}{$s_{\theta}=30$, $s_{\beta}=30$}\tabularnewline
	& \multicolumn{2}{c}{MSE} & \multicolumn{2}{c}{CP} &  & \multicolumn{2}{c}{MSE} & \multicolumn{2}{c}{CP}\tabularnewline
	AIPW & 0&207 & 0&934 &  & 0&204 & 0&936\tabularnewline
	ARB & 0&105 & 0&942 &  & 0&122 & 0&928\tabularnewline
	SDR & 0&057 & 0&992 &  & 0&060 & 0&996\tabularnewline
	& \multicolumn{2}{c}{} & \multicolumn{2}{c}{} &  & \multicolumn{2}{c}{} & \multicolumn{2}{c}{}\tabularnewline
\end{tabular}
\end{center}
\end{table}

\newpage
\appendix

 \centerline{\LARGE \bf Supplementary Materials}
  \vskip 20pt 
  
 Supplementary materials collect details of the main results and proofs.
 \vskip 20pt

{ \bf{Notations:}}
In the rest of the paper, we shall use the following notations.  We will frequently use the function 
$$q(z):=1+\exp(-z) $$
 as well as $\dot{q}(z)=d q(z)/dz=-\exp(-z) $. We will use the notations $H_A$ and $H_B$ to denote $\Jcal$ and $\Jcal^c$, respectively; doing so allows us to easily see the symmetry between $\Jcal$ and $\Jcal^c$. Moreover, $b_n=n/2$,   $b_{n}^{-1}\sum_{i\in H_{A}}$ and $b_{n}^{-1}\sum_{i\in H_{B}}$ will be denoted by $\EE_{n,H_{A}}$ and $\EE_{n,H_{B}}$. Here, we assume that  $n$ is an even number so $b_n$ is an integer. 
 
 We define $\mathbb{S}^{p-1}=\{v\in\mathbb{R}^{p}:\ \|v\|_{2}=1\}$. For any $k_0>0$ and any $J\subseteq \{1,...,p\} $, we define  the cone set $\mathcal{C}(J,k_{0})=\left\{ x\in\mathbb{R}^{p}:\ \|x_{J^{c}}\|_{1}\leq k_{0}\|x_{J}\|_{1}\right\} $. We use $'$ to denote the transpose.

\section{Proof of Theorem \ref{thm: main}}

We  notice that the KKT condition  for $\hbeta_{(w),F} $ defined in (\ref{eq: est beta}) reads 
\begin{equation}
\left\Vert \EE_{n,H_{F}}W_{i}\dot{q}(X_{i}'\hat{\theta}_{(1),F})(X_{i}'\hat{\beta}_{(1),F}-Y_{i})X_{i}\right\Vert _{\infty}\leq\lambda_{\beta}/4\quad{\rm for}\ F\in\{\Jcal,\Jcal^c \}=\{A,B\},\label{eq: KKT beta}
\end{equation}

Moreover, we notice that the solution for Algorithm \ref{alg1} always satisfies
\begin{equation}
\left\Vert \EE_{n,H_{F}}\left[1-W_{i}q(X_{i}'\hat{\theta}_{(1),F})\right]X_{i}\right\Vert _{\infty}\leq\lambda_{\theta}\quad{\rm for}\ F\in\{\Jcal,\Jcal^c \}=\{A,B\}\label{eq: KKT theta}
\end{equation}

Moreover, under Assumptions \ref{assu: RE condition} and \ref{assu: unconfoundedness}, we can define constants $M_1,...,M_5>0$ such that the following conditions hold (which we will establish as a reasonable later on):

	\begin{enumerate}
		\item $\PP(\|X\|_{\infty}\leq M_{1})=1$ and $\|X_{i}\|_{\psi_{2}}\leq M_{1}$
		for some constant $M_{1}>0$.
		\item $\EE X_{i}X_{i}'W_{i}$, $\EE X_{i}X_{i}'(1-W_{i})$, $\EE X_{i}X_{i}'\exp(-X_{i}'\theta_{(1)})W_{i}$
		have all the eigenvalues in a fixed interval $[M_{2},M_{3}]$, where
		$M_{2},M_{3}>0$ are constants. 
		\item For $j\in\{0,1\}$, $\EE(\varepsilon_{i,(j)}\mid X_{i})=0$ and there
		exists a constant $M_{4}>0$ such that $\EE(\exp(t\varepsilon_{i,(j)})\mid X_{i})\leq\exp(M_{4}t^{2})$
		$\forall t\in\mathbb{R}$. 
		\item $\|\theta_{(1)}\|_{1}\leq M_{5}$ for some constant $M_{5}>0$ and $\kappa_{0}\geq M_{5}$ is suitably chosen. 
	\end{enumerate}

\subsection{Main results}
Main body of the proof consists of three big components. Theorem \ref{thm: sparse beta} showcases the asymptotic normality result whenever the outcome model is  ultra-sparse.  
Theorem \ref{thm: sparse theta} showcases the result of ultra-sparse propensity  model. Theorem \ref{thm: main} is then completed with the help of Lemma \ref{lem: asym var} and Theorem \ref{thm: consistency estim var} establishing consistency of estimation of the asymptotic variance.

\begin{thm}[Ultra-sparse outcome model]
	\label{thm: sparse beta}Let Assumptions \ref{assu: RE condition}
	and \ref{assu: unconfoundedness} hold. Suppose that $\|\beta_{(1)}\|_{0}=o(\sqrt{n}/\log p)$
	and $\|\theta_{(1)}\|_{0}=o(n/\log p)$. Then 
	\[
	\sqrt{n}(\hat{\mu}_{(1)}-\mu_{(1)})=n^{-1/2}\sum_{i=1}^{n}\left[W_{i}\varepsilon_{i,(1)}q(X_{i}'\theta_{(1)})+\left(X_{i}'\beta_{(1)}-\mu_{(1)}\right)\right]+o_{P}(1).
	\]
\end{thm}

\begin{proof}
	See Section \ref{sec: proof thm sparse beta}.
\end{proof}

\begin{thm}[Ultra-sparse propensity model]
	\label{thm: sparse theta}Let Assumptions \ref{assu: RE condition}
	and \ref{assu: unconfoundedness} hold. Suppose that $\|\theta_{(1)}\|_{0}=o(\sqrt{n}/\log p)$
	and $\|\beta_{(1)}\|_{0}=O(n^{3/4}/\log p)$. Then
	\[
	\sqrt{n}(\hat{\mu}_{(1)}-\mu_{(1)})=n^{-1/2}\sum_{i=1}^{n}\left[W_{i}\varepsilon_{i,(1)}q(X_{i}'\theta_{(1)})+\left(X_{i}'\beta_{(1)}-\mu_{(1)}\right)\right]+o_{P}(1).
	\]
\end{thm}

\begin{proof}
	See Section \ref{sec: proof thm sparse theta}.
\end{proof}

\begin{lem}
	\label{lem: asym var} Consider $V_*$ defined in Theorem \ref{thm: main}. Under the conditions of Theorem \ref{thm: main},
	we have 
	\[
	V_{*}=\EE W_{i}\varepsilon_{i,(1)}^{2}q^2(X_{i}'\theta_{(1)})+\EE(1-W_{i})\varepsilon_{i,(0)}^{2}q^2(X_{i}'\theta_{(0)})+\EE(X_{i}'(\beta_{(1)}-\beta_{(0)})-\tau)^{2}.
	\]
\end{lem}
\begin{proof}
	See Section \ref{sec: proof lem asy var}.
\end{proof}

\begin{thm}
	\label{thm: consistency estim var}Under the conditions of Theorem
	\ref{thm: main}, we have $\hat{V}=V_{*}+o_{P}(1)$. 
\end{thm}

\begin{proof}
	See Section \ref{sec: proof consistency var estimate}.
\end{proof}

\subsection{Auxiliary results}

Proofs of the main results require a sequence of  statements discussing properties of the newly proposed estimators.

\subsubsection{Estimators for the propensity score's $\theta_{(1)}$}

First we discuss the eigenvalue properties of various design matrices.
\begin{lem}
	\label{lem: RE condition}Let Assumptions \ref{assu: RE condition}
	and \ref{assu: unconfoundedness} hold. 
	Assume that $s=o(n/\log p)$. Then with probability approaching one,
	we have that
	\begin{enumerate}
		\item $\EE_{n,H_{A}}W_{i}(X_{i}'v)^{2}\geq c_{1}^{2}\|v_{J}\|_{2}^{2}$ for
		all $v\in\bigcup_{|J|\leq s}\mathcal{C}(J,3)$ and $v\in\bigcup_{|J|\leq s}\mathcal{C}(J,1)$
		\item $\EE_{n,H_{A}}(X_{i}'v)^{2}/\|v\|_{2}^{2}\leq c_{2}^{2}$ for all $v\in\bigcup_{|J|\leq s}\mathcal{C}(J,3)$. 
		\item $\|\EE_{n,H_{A}}X_{i}(W_{i}q(X_{i}'\theta_{(1)})-1)\|_{\infty}\leq0.5\lambda_{\theta}$
		for suitably chosen $\lambda_{\theta}\asymp\sqrt{n^{-1}\log p}$
		\item $\|\EE_{n,H_{A}}X_{i}W_{i}\exp(X_{i}'\hat{\theta}_{(1),A})\varepsilon_{i,(1)}\|_{\infty}\leq\lambda_{\beta}/4$
		for suitably chosen $\lambda_{\beta}\asymp\sqrt{n^{-1}\log p}$, 
	\end{enumerate}
	where $c_{1},c_{2},c_{3}>0$ are constants depending only on $M_{1},...,M_{5}$.
	Analogous results hold if we replace $\EE_{n,H_{A}}$ with $\EE_{n,H_{B}}$. 
\end{lem}
\begin{proof}[Proof of Lemma \ref{lem: RE condition}]

\ 

	\noindent \textbf{Proof of the first two claims.} We invoke Theorem 16 of \citet{rudelson2013reconstruction}.
	Let $\Sigma_{1}=\EE W_{i}X_{i}X_{i}'$, $\tilde{X}_{i}=W_{i}X_{i}\Sigma_{1}^{-1/2}$
	and $k_{0}\in\{1,3\}$. Notice that $\tilde{X}_{i}$ is isotropic
	by definition. By the sub-Gaussian property of $X_{i}$ and the assumption
	that eigenvalues of $\Sigma_{1}$ are bounded away from zero, it follows
	that $\tilde{X}_{i}$ also has bounded sub-Gaussian norm. For a fixed
	$\delta\in(0,1)$, we define $d(3k_{0},A)$ as in Theorem 16 of \citet{rudelson2013reconstruction},
	where $A=\Sigma_{1}^{1/2}$. Clearly, $d(3k_{0},A)\asymp s$ and $m\lesssim s$.
	Observe that Equation (39) therein holds because $n\gg s\log p$. 
	
	Therefore, by
	their Theorem 16,  with probability approaching one, 
	\[
	1-\delta\leq\frac{\|\tilde{X}v\|_{2}/\sqrt{n}}{\|\Sigma_{1}^{1/2}v\|_{2}}\leq1+\delta, \qquad\forall v\in\bigcup_{|J|\leq s}\mathcal{C}(J,k_{0}),
	\]
	where $\tilde{X}=(\tilde{X}_{1},...,\tilde{X}_{n})'\in\mathbb{R}^{n\times p}$.
	Since $\|\Sigma_{1}^{1/2}v\|_{2}/\|v\|_{2}$ is bounded away from
	zero and infinity, it follows that there exist constants $C_{1},C_{2}>0$
	such that 
	\[
	C_{1}\leq\frac{\|\tilde{X}v\|_{2}/\sqrt{n}}{\|v\|_{2}}\leq C_{2}, \qquad\forall v\in\bigcup_{|J|\leq s}\mathcal{C}(J,k_{0}).
	\]
	
	Since $\|v\|_{2}\geq\|v_{J}\|_{2}$, we have that with probability
	approaching one, 
	\[
	C_{1}\leq\frac{\|\tilde{X}v\|_{2}/\sqrt{n}}{\|v_{J}\|_{2}}\qquad{\rm and}\qquad\frac{\|\tilde{X}v\|_{2}/\sqrt{n}}{\|v\|_{2}}\leq C_{2}\qquad\forall v\in\bigcup_{|J|\leq s}\mathcal{C}(J,k_{0}).
	\]
	
	This proves the first claim. The second claim follows by replacing
	$\tilde{X}$ with $X$. 
	
	\vskip 10pt 
	
	\noindent \textbf{Proof of the third claim.} Notice that $W_{i}q(X_{i}'\theta_{(1)})-1=v_{i,(1)}q(X_{i}'\theta_{(1)})$.
	Thus, 
	\[
	\EE_{n,H_{A}}X_{i}(W_{i}q(X_{i}'\theta_{(1)})-1)=\EE_{n,H_{A}}X_{i}q(X_{i}'\theta_{(1)})v_{i,(1)}.
	\]
	
	Notice that conditional on $\{X_{i}\}_{i\in H_{A}}$, $\{v_{i,(1)}\}_{i\in H_{A}}$
	is independent across $i$ with mean zero. Moreover, $v_{i,(1)}$
	is also sub-Gaussian since it is bounded by 2. (To see this, simply
	notice that $v_{i,(1)}=W_{i}-e_{(1)}(X_{i})$ and both $W_{i}$ and
	$e_{(1)}(X_{i})$ are bounded by 1.) By Hoeffding's inequality (e.g.,
	Proposition 5.10 in \citet{CBO9780511794308A012}), it follows that
	for $j\in\{1,...,p\}$ and for any $t>0$, 
	\[
	\PP\left(\left|\sum_{i\in H_{A}}X_{i,j}q(X_{i}'\theta_{(1)})v_{i,(1)}\right|>t\mid\{X_{i}\}_{i\in H_{A}}\right)\leq\exp\left(1-\frac{C_{3}t^{2}}{\sum_{i\in H_{A}}X_{i,j}^{2}[q(X_{i}'\theta_{(1)})]^{2}}\right),
	\]
	where $C_{3}>0$ is a universal constant. 
	
	Since $\|X\|_{\infty}\leq M_{1}$,
	$\|\theta_{(1)}\|_{1}\leq M_{5}$ and $q(X_{i}'\theta_{(1)})=1+\exp(-X_{i}'\theta_{(1)})$,
	we have that 
	$$\sum_{i\in H_{A}}X_{i,j}^{2}[q(X_{i}'\theta_{(1)})]^{2}\leq b_{n}C_{4},$$
	where $C_{4}=M_{1}^{2}(1+\exp(M_{1}M_{5}))^{2}$. Hence, by the union
	bound, it follows that 
	\[
	\PP\left(\max_{1\le j\leq p}\left|\sum_{i\in H_{A}}X_{i,j}q(X_{i}'\theta_{(1)})v_{i,(1)}\right|>t\right)\leq p\exp\left(1-\frac{C_{3}t^{2}}{b_{n}C_{4}}\right).
	\]
	
	Hence, by taking $\lambda_{\theta}=2\sqrt{2b_{n}C_{3}^{-1}C_{4}\log p}$,
	we have 
	$$\PP(\|\EE_{n,H_{A}}X_{i}(W_{i}q(X_{i}'\theta_{(1)})-1)\|_{\infty}>0.5\lambda_{\theta})=\exp(1)/p\rightarrow0.$$
	This proves the third claim. 
	
	\vskip 10pt 
	
\noindent	\textbf{Proof of the fourth claim.} The argument is essentially the
	same as the proof of the third claim. We outline the strategy. Notice
	that $\hat{\theta}_{(1),A}$ is computed using $\{(X_{i},W_{i})\}_{i\in H_{A}}$,
	which depends only on $\{(X_{i},v_{i,(1)})\}_{i\in H_{A}}$. Since
	$\varepsilon_{i,(1)}$ and $v_{i,(1)}$ are independent conditional
	on $X_{i}$, it follows that condition on $\{(X_{i},W_{i})\}_{i\in H_{A}}$,
	$\{\varepsilon_{i,(1)}\}_{i\in H_{A}}$ is independent across $i$
	with mean zero. Therefore, exactly the same argument as above with
	$v_{i,(1)}$ replaced by $\varepsilon_{i,(1)}$ would yield the fourth
	claim. The proof is complete. 
\end{proof}

\begin{lem}
	\label{lem: convex loss}For any $x\in\mathbb{R}$, $\exp(-x)-1+x\geq0.4x^{2}-0.1x^{3}$. 
\end{lem}
\begin{proof} [Proof of Lemma \ref{lem: convex loss}]
	Let $h(x)=\exp(-x)-1+x-0.4x^{2}+0.1x^{3}$. Then $\ddot{h}(x)=d^{2}h(x)/dx^{2}=\exp(-x)+0.6x-0.8$.
	We first show that $\ddot{h}(\cdot)$ is convex and then derive the
	minimum of $h(\cdot)$. 
	
	By taking the second derivative of $\ddot{h}(\cdot)$, we can see
	that $\ddot{h}(\cdot)$ is convex. To find the minimum of $\ddot{h}(\cdot)$,
	we consider the first order condition: $-\exp(-x)+0.6=0$, i.e., 
	$$\arg\min_{x\in\mathbb{R}}h(x)=\log(5/3).$$
	This means that $\min_{x\in\mathbb{R}}\ddot{h}(x)=\ddot{h}(\log(5/3))=0.6+0.6\times\log(5/3)-0.8>0$.
	Therefore, $\ddot{h}(\cdot)$ is non-negative, which means that $h(\cdot)$
	is convex.
	
	Now we take the first order condition for $\min_{x\in\mathbb{R}}h(x)$,
	leading to 
	$$-\exp(-x)+1-0.8x+0.3x^{2}=0.$$ Clearly, $x=0$ is a solution.
	Since $h(\cdot)$ is convex, this is the only solution. Therefore,
	$\min_{x\in\mathbb{R}}h(x)=h(0)=0$. The proof is complete. 
\end{proof}

\subsubsection{Lasso-type estimator: $\check \theta$}

For the next result, we introduce a simplified notation to help with the exposition.

Define 
$$L_{n}(\theta)=\EE_{n,H_{A}}[(1-W_{i})X_{i}'\theta+W_{i}\exp(-X_{i}'\theta)]$$
	and $\dot{L}_{n}(\theta)=\EE_{n,H_{A}}(1-W_{i}q(X_{i}'\theta))X_{i}$.
	Define 
	$$\check{\theta}_{(1),A}=\arg\min_{\theta} \left\{ L_{n}(\theta)+\lambda\|\theta\|_{1}\right\}.$$
		Let $J\subset\{1,...,p\}$ satisfy $\text{supp}(\theta_{(1)})\subset J$.
	
\begin{prop}
	\label{prop: logistic machinery}
	We assume that $\|\dot{L}_{n}(\theta_{(1)})\|_{\infty}\leq c\lambda$,
	$\EE_{n,H_{A}}W_{i}\exp(-X_{i}'\theta)(X_{i}\delta)^{2}\geq\kappa_{1}\|\delta_{J}\|_{2}^{2}$,
	$\EE_{n,H_{A}}W_{i}\exp(-X_{i}'\theta)|X_{i}'\delta|^{3}\leq\kappa_{2}\|\delta_{J}\|_{2}^{3}$
	and $\lambda^{2}s\leq\kappa_{1}^{4}\kappa_{2}^{-2}/20$ for any $\delta\in\{v:\ \|v_{J}\|_{1}\leq\|v_{J^{c}}\|_{1}(1+c)/(1-c)\}$,
	where $s=|J|$. Let $\Delta=\check{\theta}_{(1),A}-\theta_{(1)}$.
	
	Then, 
	$$\|\Delta_{J^{c}}\|_{1}\leq(1-c)^{-1}(1+c)\|\Delta_{J}\|_{1},$$
	
	$$\|\Delta_{J}\|_{2}\leq10\kappa_{1}^{-1}\lambda\sqrt{s}, \qquad  \mbox{and} \qquad   \|\Delta\|_{1}\leq20(1-c)^{-1}\kappa_{1}^{-1}\lambda s.$$ 
\end{prop}
\begin{proof}[Proof of Proposition \ref{prop: logistic machinery}]
	Our proof uses the minoration argument from \citet{belloni2011L1,belloni2016post}.
	We notice that 
	\[
	L_{n}(\theta_{(1)}+\Delta)+\lambda\|\theta_{(1)}+\Delta\|_{1}\leq L_{n}(\theta_{(1)})+\lambda\|\theta_{(1)}\|_{1}.
	\]
	
	Since $\|\theta_{(1)}+\Delta\|_{1}=\|\theta_{(1)}+\Delta_{J}\|_{1}+\|\Delta_{J^{c}}\|_{1}$,
	we have that 
	\begin{equation}
	L_{n}(\theta_{(1)}+\Delta)-L_{n}(\theta_{(1)})+\lambda\|\Delta_{J^{c}}\|_{1}\leq\lambda\|\theta_{(1)}\|_{1}-\lambda\|\theta_{(1)}+\Delta_{J}\|_{1}\leq\lambda\|\Delta_{J}\|_{1}.\label{eq: est theta 4}
	\end{equation}
	
	By the convexity of $L_{n}(\cdot)$, we have that $L_{n}(\theta_{(1)}+\Delta)-L_{n}(\theta_{(1)})-\Delta'\dot{L}_{n}(\theta_{(1)})\geq0$.
	Moreover, 
	\[
	|\Delta\dot{L}_{n}(\theta_{(1)})|\leq\|\Delta\|_{1}\|\dot{L}_{n}(\theta_{(1)})\|_{\infty}\leq c\lambda\|\Delta\|_{1}\leq c\lambda\|\Delta_{J}\|_{1}+c\lambda\|\Delta_{J^{c}}\|_{1}.
	\]
	
	The above two displays imply that 
	\begin{align*}
	\lambda\|\Delta_{J}\|_{1}\geq L_{n}(\theta_{(1)}+\Delta)-L_{n}(\theta_{(1)})+\lambda\|\Delta_{J^{c}}\|_{1} & \geq\Delta'\dot{L}_{n}(\theta_{(1)})+\lambda\|\Delta_{J^{c}}\|_{1}\\
	& \geq-\left(c\lambda\|\Delta_{J}\|_{1}+c\lambda\|\Delta_{J^{c}}\|_{1}\right)+\lambda\|\Delta_{J^{c}}\|_{1}.
	\end{align*}
	
	Rearranging the terms, we obtain 
	\begin{equation}
	\|\Delta_{J^{c}}\|_{1}\leq\frac{1+c}{1-c}\|\Delta_{J}\|_{1}.\label{eq: est theta 4.1}
	\end{equation}
	
	We denote $A=\{v:\ \|v_{J^{c}}\|_{1}\leq(1-c)^{-1}(1+c)\|v_{J}\|_{1}\}$.
	We define the following quantity 
	\[
	r_{A}=\sup\left\{ r>0:\ \inf_{\|\delta\|\leq r,\ \delta\in A}\ \frac{L_{n}(\theta_{(1)}+\delta)-L_{n}(\theta_{(1)})-\delta'\dot{L}_{n}(\theta_{(1)})}{\|\delta\|^{2}}\geq c_{1}\right\} ,
	\]
	where $c_{1}=0.1\kappa_{1}$. 
	
	\textbf{Step 1:} show $r_{A}\geq3\kappa_{1}\kappa_{2}^{-1}$. 
	
	Define $f:\mathbb{R}\mapsto\mathbb{R}$ by $f(x)=\exp(-x)-1+x$. Then
	we can rewrite 
	\[
	L_{n}(\theta_{(1)}+\delta)-L_{n}(\theta_{(1)})-\delta'\dot{L}_{n}(\theta_{(1)})=\EE_{n,H_{A}}W_{i}\exp(-X_{i}'\theta_{(1)})f(X'\delta).
	\]
	
	By Lemma \ref{lem: convex loss}, we have that 
	\begin{align*}
	&L_{n}(\theta_{(1)}+\delta)-L_{n}(\theta_{(1)})-\delta'\dot{L}_{n}(\theta_{(1)}) 
	\\
	&\qquad  \geq0.4\EE_{n,H_{A}}W_{i}\exp(-X_{i}'\theta_{(1)})(X_{i}\delta)^{2}-0.1\EE_{n,H_{A}}W_{i}\exp(-X_{i}'\theta_{(1)})|X_{i}'\delta|^{3}\\
	& \qquad \geq0.4\kappa_{1}\|\delta_{J}\|_{2}-0.1\kappa_{2}\|\delta_{J}\|_{2}^{3}.
	\end{align*}
	
	Thus, for any $\delta\in A$, 
	\[
	\frac{L_{n}(\theta_{(1)}+\delta)-L_{n}(\theta_{(1)})-\delta'\dot{L}_{n}(\theta_{(1)})}{\|\delta_{J}\|_{2}^{2}}\geq0.4\kappa_{1}-0.1\kappa_{2}\|\delta_{J}\|_{2}.
	\]
	
	Therefore, 
	\begin{multline*}
	\inf_{\delta\in A,\ \|\delta_{J}\|_{2}\leq3\kappa_{1}\kappa_{2}^{-1}}\ \frac{L_{n}(\theta_{(1)}+\delta)-L_{n}(\theta_{(1)})-\delta'\dot{L}_{n}(\theta_{(1)})}{\|\delta_{J}\|_{2}^{2}}\\
	\geq\inf_{\delta\in A,\ \|\delta_{J}\|_{2}\leq3\kappa_{1}\kappa_{2}^{-1}}\left(0.4\kappa_{1}-0.1\kappa_{2}\|\delta_{J}\|_{2}\right)\geq0.1\kappa_{1}=c_{1}.
	\end{multline*}

	Hence, we have that 
	\begin{equation}
	r_{A}\geq3\kappa_{1}\kappa_{2}^{-1}.\label{eq: est theta 4.4}
	\end{equation}
	
	\textbf{Step 2:} show that $\|\Delta_{J}\|_{2}\leq r_{A}$. 
	
	We proceed by contradiction. Assume that $\|\Delta_{J}\|_{2}>r_{A}$.
	For $v\in\mathbb{R}^{p}$, define 
	$$Q(v)=L_{n}(\theta_{(1)}+v)-L_{n}(\theta_{(1)})-v'\dot{L}_{n}(\theta_{(1)}).$$
	By the convexity of $L_{n}(\cdot)$, $Q(\cdot)$ is also convex. Let
	$t=r_{A}/\|\Delta_{J}\|_{2}$. Since $\|\Delta_{J}\|_{2}>r_{A}$,
	$t\in(0,1)$. Thus, 
	\[
	tQ\left(\Delta\right)+(1-t)Q(0)\geq Q(t\Delta+(1-t)\cdot0).
	\]
	
	Since $Q(0)=0$, the above display implies that 
	\begin{equation}
	Q(\Delta)\geq\frac{1}{t}Q(t\Delta)=r_{A}^{-1}\|\Delta_{J}\|_{2}Q\left(t\Delta\right)\overset{\text{(i)}}{\geq}c_{1}r_{A}\|\Delta_{J}\|_{2}.\label{eq: est theta 5}
	\end{equation}
	where (i) follows by the definition of $r_{A}$ and fact that $\|(t\Delta)_{J}\|_{2}=r_{A}$
	and $t\Delta\in A$. By (\ref{eq: est theta 4}), we have that 
	\begin{align*}
	\lambda\|\Delta_{J}\|_{1} & \geq Q(\Delta)+\Delta'\dot{L}_{n}(\theta_{(1)})+\lambda\|\Delta_{J^{c}}\|_{1}\\
	& \overset{\text{(i)}}{\geq}c_{1}r_{A}\|\Delta_{J}\|_{2}+\Delta'\dot{L}_{n}(\theta_{(1)})+\lambda\|\Delta_{J^{c}}\|_{1}\\
	& \overset{\text{(ii)}}{\geq}c_{1}r_{A}\|\Delta_{J}\|_{2}-c\lambda\|\Delta\|_{1}+\lambda\|\Delta_{J^{c}}\|_{1},
	\end{align*}
	where (i) follows by (\ref{eq: est theta 5}) and (ii) follows by
	$|\Delta\dot{L}_{n}(\theta_{(1)})|\leq\|\Delta\|_{1}\|\dot{L}_{n}(\theta_{(1)})\|_{\infty}\leq c\lambda\|\Delta\|_{1}$.
	Rearranging the terms, we obtain 
	\[
	c_{1}r_{A}\|\Delta_{J}\|_{2}+(1-c)\lambda\|\Delta_{J^{c}}\|_{1}\leq\lambda\|\Delta_{J}\|_{1}.
	\]
	
	Therefore, 
	\[
	r_{A}\leq\frac{\lambda\|\Delta_{J}\|_{1}}{c_{1}\|\Delta\|_{2}}\overset{\text{(i)}}{\leq}\frac{\lambda\sqrt{s}\|\Delta_{J}\|_{2}}{c_{1}\|\Delta\|_{2}}\leq c_{1}^{-1}\lambda\sqrt{s},
	\]
	where (i) follows by Hölder's inequality.
	
	By (\ref{eq: est theta 4.4}), $c_{1}=0.1\kappa_{1}$ and $r_{A}\geq3\kappa_{1}\kappa_{2}^{-1}$.
	Hence, we have $3\kappa_{1}\kappa_{2}^{-1}\leq10\kappa_{1}^{-1}\lambda\sqrt{s}$,
	which means $\lambda^{2}s\geq0.09\kappa_{1}^{4}\kappa_{2}^{-2}$.
	This contradicts the assumption of $\lambda^{2}s\leq\kappa_{1}^{4}\kappa_{2}^{-2}/20$.
	Hence, 
	\begin{equation}
	\|\Delta_{J}\|_{2}\leq r_{A}.\label{eq: est theta 4.5}
	\end{equation}
	
	\textbf{Step 3:} derive the desired result. 
	
	By (\ref{eq: est theta 4}), we have that 
	\begin{align*}
	\lambda\|\Delta_{J}\|_{1} & \geq L_{n}(\theta_{(1)}+\Delta)-L_{n}(\theta_{(1)})-\Delta'\dot{L}_{n}(\theta_{(1)})+\Delta'\dot{L}_{n}(\theta_{(1)})+\lambda\|\Delta_{J^{c}}\|_{1}\\
	& \overset{\text{(i)}}{\geq}c_{1}\|\Delta_{J}\|_{2}^{2}+\Delta'\dot{L}_{n}(\theta_{(1)})+\lambda\|\Delta_{J^{c}}\|_{1}\\
	& \overset{\text{(ii)}}{\geq}c_{1}\|\Delta_{J}\|_{2}^{2}-c\lambda\|\Delta\|_{1}+\lambda\|\Delta_{J^{c}}\|_{1},
	\end{align*}
	where (i) follows by $\|\Delta_{J}\|_{2}\leq r_{A}$ (due to (\ref{eq: est theta 4.5}))
	and (ii) follows by $|\Delta\dot{L}_{n}(\theta_{(1)})|\leq\|\Delta\|_{1}\|\dot{L}_{n}(\theta_{(1)})\|_{\infty}\leq c\lambda\|\Delta\|_{1}$.
	Rearranging the terms, we obtain $c_{1}\|\Delta_{J}\|_{2}^{2}+(1-c)\lambda\|\Delta_{J^{c}}\|_{1}\leq\lambda\|\Delta_{J}\|_{1}$.
	Hence, 
	\[
	\|\Delta_{J}\|_{2}\leq\frac{\lambda\|\Delta_{J}\|_{1}}{c_{1}\|\Delta_{J}\|_{2}}\overset{\text{(i)}}{\leq}\frac{\lambda\|\Delta_{J}\|_{2}\sqrt{s}}{c_{1}\|\Delta_{J}\|_{2}}=c_{1}^{-1}\lambda\sqrt{s}\overset{\text{(ii)}}{=}10\kappa_{1}^{-1}\lambda\sqrt{s},
	\]
	where (i) follows by Hölder's inequality and (ii) follows by $c_{1}=0.1\kappa_{1}$.
	Thus, 
	\[
	\|\Delta\|_{1}=\|\Delta_{J}\|_{1}+\|\Delta_{J^{c}}\|_{1}\overset{\text{(i)}}{\leq}\left(1+\frac{1+c}{1-c}\right)\|\Delta_{J}\|_{1}\overset{\text{(ii)}}{\leq}\frac{2}{1-c}\|\Delta_{J}\|_{2}\sqrt{s}=\frac{20}{1-c}\kappa_{1}^{-1}\lambda s,
	\]
	where (i) follows by (\ref{eq: est theta 4.1}) and (ii) follows by
	Hölder's inequality.
\end{proof}

With the help of Proposition \ref{prop: logistic machinery} we are now able to establish estimation quality properties of the introduced estimator of $\theta_{(1)}$.
\begin{lem}
	\label{lem: lasso bnd theta-1}Let Assumptions \ref{assu: RE condition}
	and \ref{assu: unconfoundedness} hold. Assume that $\|\theta_{(1)}\|_{0}=o(n/\log p)$.
	Then 
	$$\PP\left(\check{\theta}_{(1),A}-\theta_{(1)}\in\mathcal{C}({\rm supp}(\theta_{(1)}),3)\right)\rightarrow1,$$
	$$\|\check{\theta}_{(1),A}-\theta_{(1)}\|_{1}=O_{P}\left(s_{\theta}\sqrt{n^{-1}\log p}\right)$$
	and 
	$$\|\check{\theta}_{(1),A}-\theta_{(1)}\|_{2}=O_{P}\left(\sqrt{s_{\theta}n^{-1}\log p}\right).$$ 
\end{lem}
\begin{proof}[Proof of Lemma \ref{lem: lasso bnd theta-1}]
	We define the event
	\begin{multline*}
	\mathcal{M}=\left\{ \|\tilde{\theta}_{(1),A}\|_{1}\leq M_{5}\right\} \bigcap\left\{ \|\EE_{n,H_{A}}X_{i}(1-W_{i}q(X_{i}'\theta_{(1)}))\|_{\infty}\leq0.5\lambda_{\theta}\right\} \\
	\bigcap\left\{ \min_{|J|\leq2s_{\theta}}\inf_{\|v_{J^{c}}\|_{1}\leq\|v_{J}\|_{1}}\frac{\EE_{n,H_{A}}W_{i}(X_{i}'v)^{2}}{\|v_{J}\|_{2}^{2}}\geq c_{1}^{2}\right\} 
	\\
	\bigcap\left\{ \max_{|J|\leq2s_{\theta}}\max_{\|v_{J^{c}}\|_{1}\leq\|v_{J}\|_{1}}\frac{\EE_{n,H_{A}}(X_{i}'v)^{2}}{\|v\|_{2}^{2}}\leq c_{2}^{2}\right\} 
	\end{multline*}
	where $c_{1},c_{2}>0$ are constants from Lemma \ref{lem: RE condition}.
	By Lemma \ref{lem: RE condition}, $\PP(\mathcal{M})\rightarrow1$. 
	
	Let $B=\text{supp}(\theta_{(1)})$. We apply Proposition \ref{prop: logistic machinery}
	with $J=B$, $c=0.5$ and $\lambda=\lambda_{\theta}$, obtaining
	that on the event $\mathcal{M}$, $\|\Delta_{B^{c}}\|_{1}\leq3\|\Delta_{B}\|_{1}$
	and 
	\begin{equation}
	\|\Delta\|_{1}\leq40\kappa_{1}^{-1}\lambda_{\theta}s_{\theta}.\label{eq: est theta rate 3-1}
	\end{equation}
	
	Let $N_{0}$ denote the $s_{\theta}$ indices in $B^{c}$ corresponding
	to the largest $s_{\theta}$ entries (in absolute value) of $\Delta$.
	Let $N=B\bigcup N_{0}$. We now apply Lemma 6.9 of \citet{buhlmann2011statistics}
	to the vector $\Delta_{B^{c}}$. Once we exclude the largest $s_{\theta}$
	entries (in magnitude) in $\Delta_{B^{c}}$, we obtain $\Delta_{N^{c}}$.
	Hence, Lemma 6.9 of \citet{buhlmann2011statistics} implies that 
	\[
	\|\Delta_{N^{c}}\|_{2}\leq s_{\theta}^{-1/2}\|\Delta_{B^{c}}\|_{1}\leq s_{\theta}^{-1/2}\|\Delta\|_{1}\overset{\text{(i)}}{\leq}40\kappa_{1}^{-1}\lambda_{\theta}\sqrt{s_{\theta}},
	\]
	where (i) follows by (\ref{eq: est theta rate 3-1}). 
	
	Since $|N|\leq2s_{\theta}$ (due to the definition of $N$), $\|\Delta_{N^{c}}\|_{1}\leq\|\Delta_{N}\|_{1}$
	on the event $\mathcal{M}$. We now apply Proposition \ref{prop: logistic machinery}
	with $J=N$, $c=0.5$ and $\lambda=\lambda_{\theta}$, obtaining
	that on the event $\mathcal{M}$, 
	\[
	\|\Delta_{N}\|_{2}\leq10\kappa_{1}^{-1}\lambda\sqrt{|N|}\overset{\text{(i)}}{\leq}10\kappa_{1}^{-1}\lambda_{\theta}\sqrt{2s_{\theta}},
	\]
	where (i) follows by $|N|\leq2s_{\theta}$. Hence, the above two displays
	imply 
	\[
	\|\Delta\|_{2}^{2}=\|\Delta_{N}\|_{2}^{2}+\|\Delta_{N^{c}}\|_{2}^{2}\leq400(1-c)^{-2}\kappa_{1}^{-2}\lambda^{2}s_{\theta}+200\kappa_{1}^{-2}\lambda^{2}s_{\theta}.
	\]
	
	Hence, $\|\Delta\|_{2}=O_{P}(\sqrt{s_{\theta}n^{-1}\log p})$. 
\end{proof}

\subsubsection{Dantzig-type estimator: $\hat \theta$}
\begin{lem}
	\label{lem: quasi convex exp}For any $z>0$, there exists a constant
	$C_{z}>0$ depending only on $z$ such that for any $x\in[-z,z]$,
	$(1-\exp(-x))x\geq C_{z}x^{2}$. 
\end{lem}
\begin{proof}[Proof of Lemma \ref{lem: quasi convex exp}]
	Let $f(x)=[(1-\exp(-x))x]/x^{2}$. Then 
	$$df(x)/dx=x^{-1}\exp(-x)(1+x-\exp(x)).$$
	By the elementary inequality of $\exp(x)\geq1+x$ for any $x\in\mathbb{R}$,
	we have that $df(x)/dx\leq0$ for any $x\in\mathbb{R}$. Thus, $f(x)$
	is non-increasing on $\mathbb{R}$. Hence, $\inf_{x\in[-z,z]}f(x)=f(z)$.
	Let $C_{z}=f(z)$. Then $f(x)\geq C_{z}$, which implies $(1-\exp(-x))x\geq C_{z}x^{2}$.
	The proof is complete. 
\end{proof}
\begin{prop}
	\label{prop: dantzig logistic}Suppose that $\|\EE_{n,H_{A}}X_{i}(1-W_{i}q(X_{i}'\theta_{(1)}))\|_{\infty}\leq\lambda$
	and\\
	 $$\inf_{\delta:\ \|\delta_{J^{c}}\|_{1}\leq\|\delta_{J}\|_{1}}\EE_{n,H_{A}}W_{i}\exp(-X_{i}'\theta_{(1)})(X_{i}\delta)^{2}/\|\delta_{J}\|_{2}^{2}\geq\kappa_{1},$$
	where $J\subset\{1,...,p\}$ satisfies ${\rm supp}(\theta_{(1)})\subseteq J$.
	Let $\Delta=\tilde{\theta}_{(1),A}-\theta_{(1)}$. 
	
	Then, 
	$$\|\Delta_{J^{c}}\|_{1}\leq\|\Delta_{J}\|_{1}, \qquad 
	\|\Delta\|_{1}\leq8D_{1}^{-1}\kappa_{1}^{-1}\lambda s, \qquad  \mbox{and}, \qquad  \|\Delta_{J}\|_{2}\leq4D_{1}^{-1}\kappa_{1}^{-1}\lambda\sqrt{s},$$
	where $s=|J|$ and $D_{1}>0$ is a constant depending only on $M_{1}$
	and $M_{5}$. 
\end{prop}
\begin{proof}[Proof of Proposition \ref{prop: dantzig logistic}]
	Since $\|\tilde{\theta}_{(1),A}\|_{1}\leq\|\theta_{(1)}\|_{1}$ and
	$\|\tilde{\theta}_{(1),A}\|_{1}=\|\theta_{(1)}+\Delta_{J}\|_{1}+\|\Delta_{J^{c}}\|_{1}$
	(due to ${\rm supp}(\theta_{(1)})\subseteq J$), we have that $\|\Delta_{J^{c}}\|_{1}\leq\|\Delta_{J}\|_{1}$.
	Hence, $\|\Delta\|_{1}=\|\Delta_{J}\|_{1}+\|\Delta_{J^{c}}\|_{1}\leq2\|\Delta_{J}\|_{1}$. 
	
	By construction, we have $\|\EE_{n,H_{A}}X_{i}(1-W_{i}q(X_{i}'(\theta_{(1)}+\Delta)))\|_{\infty}\leq\lambda$.
	Therefore, 
	$$\|\EE_{n,H_{A}}X_{i}W_{i}(q(X_{i}'(\theta_{(1)}+\Delta))-q(X_{i}'\theta_{(1)}))\|_{\infty}\leq2\lambda.$$
	This means 
	\[
	\left\Vert \EE_{n,H_{A}}X_{i}W_{i}\exp(-X_{i}'\theta_{(1)})[1-\exp(-X_{i}'\Delta)]\right\Vert _{\infty}\le2\lambda.
	\]
	
	Let $\phi(x)=(1-\exp(-x))x$. Therefore, 
	\begin{align*}
	&\EE_{n,H_{A}}W_{i}\exp(-X_{i}'\theta_{(1)})\phi(X_{i}'\Delta)\\
	 &\qquad\leq\|\Delta\|_{1}\left\Vert \EE_{n,H_{A}}X_{i}W_{i}\exp(-X_{i}'\theta_{(1)})[1-\exp(-X_{i}'\Delta)]\right\Vert _{\infty}\le2\lambda\|\Delta\|_{1}.
	\end{align*}
	
	Notice that by construction $\|\tilde{\theta}_{(1),A}\|_{1}\leq\|\theta_{(1)}\|_{1}$.
	Since $\|\theta_{(1)}\|_{1}\leq M_{5}$ and $\|X\|_{\infty}\leq M_{1}$
	by assumption, $\|X\Delta\|_{\infty}\leq2M_{1}M_{4}$ is also bounded.
	By Lemma \ref{lem: quasi convex exp}, there exists a constant $D_{1}>0$
	depending only on $M_{1}M_{5}$ such that $\phi(X_{i}'\Delta)\geq D_{1}(X_{i}'\Delta)^{2}$.
	It follows that 
	\[
	2\lambda\|\Delta\|_{1}\geq D_{1}\EE_{n,H_{A}}W_{i}\exp(-X_{i}'\theta_{(1)})(X_{i}'\Delta)^{2}\overset{\text{(i)}}{\geq}D_{1}\kappa_{1}\|\Delta_{J}\|_{2}^{2},
	\]
	where (i) follows by $\|\Delta_{J^{c}}\|_{1}\leq\|\Delta_{J}\|_{1}$
	and the assumption of\\
	 $\inf_{\delta:\ \|\delta_{J^{c}}\|_{1}\leq\|\delta_{J}\|_{1}}\EE_{n,H_{A}}W_{i}\exp(-X_{i}'\theta_{(1)})(X_{i}\delta)^{2}/\|\delta_{J}\|_{2}^{2}\geq\kappa_{1}$. 
	
	Since $\|\Delta\|_{1}\leq2\|\Delta_{J}\|_{1}\leq2\sqrt{s}\|\Delta_{J}\|_{2}$,
	we have that 
	$$4\lambda\sqrt{s}\|\Delta\|_{2}\geq D_{1}\kappa_{1}\|\Delta_{J}\|_{2}^{2},$$
	which implies $\|\Delta_{J}\|_{2}\leq4D_{1}^{-1}\kappa_{1}^{-1}\lambda\sqrt{s}$
	and thus $\|\Delta\|_{1}\leq2\sqrt{s}\|\Delta_{J}\|_{2}\leq8D_{1}^{-1}\kappa_{1}^{-1}\lambda s$.
	The proof is complete. 
\end{proof}

With the help of Proposition \ref{prop: dantzig logistic} we are now able to complete the proof regarding the dantzig-type estimator as defined in Algorithm \ref{alg1}.
\begin{lem}
	\label{lem: DS theta bnd}Let Assumptions \ref{assu: RE condition}
	and \ref{assu: unconfoundedness} hold. Assume that $\|\theta_{(1)}\|_{0}=o(n/\log p)$.
	
	Then, $\|\tilde{\theta}_{(1),A}\|_{1}\leq M_{5}$ and $\tilde{\theta}_{(1),A}-\theta_{(1)}\in\mathcal{C}({\rm supp}(\theta_{(1)}),1)$
	with probability approaching one. 
	
	Moreover, $$\|\tilde{\theta}_{(1),A}-\theta_{(1)}\|_{1}=O_{P}(s_{\theta}\sqrt{n^{-1}\log p}),$$
	$$\|\tilde{\theta}_{(1),A}-\theta_{(1)}\|_{2}=O_{P}(\sqrt{s_{\theta}n^{-1}\log p})$$
	and
	 $$\sum_{i\in H_{A}}W_{i}\left(q(X_{i}'\tilde{\theta}_{(1),A})-q(X_{i}'\theta_{(1)})\right)^{2}=o_{P}(n),$$
	where $B={\rm supp}(\theta_{(1)})$. 
\end{lem}
\begin{proof}[Proof of Lemma \ref{lem: DS theta bnd}]
	We now combine Lemma \ref{lem: RE condition} and Proposition \ref{prop: dantzig logistic}
	to obtain the desired result. Define the event 
	\begin{multline*}
	\mathcal{M}=\left\{ \|\tilde{\theta}_{(1),A}\|_{1}\leq M_{5}\right\} \bigcap\left\{ \|\EE_{n,H_{A}}X_{i}(1-W_{i}q(X_{i}'\theta_{(1)}))\|_{\infty}\leq\lambda_{\theta}\right\} \\
	\bigcap\left\{ \min_{|J|\leq2s_{\theta}}\inf_{\|v_{J^{c}}\|_{1}\leq\|v_{J}\|_{1}}\frac{\EE_{n,H_{A}}W_{i}(X_{i}'v)^{2}}{\|v_{J}\|_{2}^{2}}\geq c_{1}^{2}\right\} \\
	\bigcap\left\{ \max_{|J|\leq2s_{\theta}}\max_{\|v_{J^{c}}\|_{1}\leq\|v_{J}\|_{1}}\frac{\EE_{n,H_{A}}(X_{i}'v)^{2}}{\|v\|_{2}^{2}}\leq c_{2}^{2}\right\} 
	\end{multline*}
	where $c_{1},c_{2}>0$ are constants from Lemma \ref{lem: RE condition}.
	By Lemma \ref{lem: RE condition}, $\PP(\mathcal{M})\rightarrow1$. 
	
	Let $\Delta=\tilde{\theta}_{(1),A}-\theta_{(1)}$ and $B={\rm supp}(\theta_{(1)})$.
	Notice that on the event $\mathcal{M}$, $\|\tilde{\theta}_{(1),A}\|_{1}\leq\|\theta_{(1)}\|_{1}\leq M_{5}$
	and $\Delta\in\mathcal{C}(B,1)$ (due to Proposition \ref{prop: dantzig logistic}).
	Since $\PP(\mathcal{M})\rightarrow1$, we have proved the first two
	claims. Now we prove the other claims in three steps. 
	
	\textbf{Step 1:} show $\|\tilde{\theta}_{(1),A}-\theta_{(1)}\|_{1}=O_{P}(s_{\theta}\sqrt{n^{-1}\log p})$.
	
	Since $\|X\theta_{(1)}\|_{\infty}\leq\|X\|_{\infty}\|\theta_{(1)}\|_{1}\leq M_{1}M_{4}$,
	we have that on the event $\mathcal{M}$, 
	\[
	\inf_{\|v_{J^{c}}\|_{1}\leq\|v_{J}\|_{1}}\frac{\EE_{n,H_{A}}W_{i}(X_{i}'v)^{2}\exp(-X_{i}'\theta_{(1)})}{\|v_{J}\|_{2}^{2}}\geq c_{1}^{2}\exp(-M_{1}M_{4}).
	\]

	We apply Proposition \ref{prop: dantzig logistic} with $J$ and obtain
	that on the event $\mathcal{M}$, 
	\begin{equation}
	\|\Delta\|_{1}\leq8D_{1}^{-1}c_{1}^{-2}\exp(M_{1}M_{4})\lambda_{\theta}s_{\theta},\label{eq: est theta rate 3}
	\end{equation}
	where $D_{1}>0$ is a constant depending only on $M_{1}$ and $M_{4}$.
	Since $\lambda_{\theta}\asymp\sqrt{n^{-1}\log p}$ and $\PP(\mathcal{M})\rightarrow1$,
	we obtain $\|\hat{\theta}_{(1),A}-\theta_{(1)}\|_{1}=O_{P}(s_{\theta}\sqrt{n^{-1}\log p})$.
	
	\textbf{Step 2:} show that $\|\tilde{\theta}_{(1),A}-\theta_{(1)}\|_{2}=O_{P}(\sqrt{s_{\theta}n^{-1}\log p})$.
	
	Let $N_{0}$ denote the $s_{\theta}$ indices in $B^{c}$ corresponding
	to the largest $s_{\theta}$ entries (in absolute value) of $\Delta$.
	Let $N=B\bigcup N_{0}$. We now apply Lemma 6.9 of \citet{buhlmann2011statistics}
	to the vector $\Delta_{B^{c}}$. Once we exclude the largest $s_{\theta}$
	entries (in magnitude) in $\Delta_{B^{c}}$, we obtain $\Delta_{N^{c}}$.
	Hence, Lemma 6.9 of \citet{buhlmann2011statistics} implies that 
	\[
	\|\Delta_{N^{c}}\|_{2}\leq s_{\theta}^{-1/2}\|\Delta_{B^{c}}\|_{1}\leq s_{\theta}^{-1/2}\|\Delta\|_{1}\leq8D_{1}^{-1}c_{1}^{-2}\exp(M_{1}M_{4})\lambda_{\theta}\sqrt{s_{\theta}},
	\]
	where (i) follows by (\ref{eq: est theta rate 3}). 
	
	We now apply Proposition \ref{prop: dantzig logistic} with $J=N$
	and obtain that on the event $\mathcal{M}$, 
	\[
	\|\Delta_{N}\|_{2}\leq4D_{1}^{-1}c_{1}^{-2}\exp(M_{1}M_{4})\lambda_{\theta}\sqrt{2s_{\theta}}.
	\]
	
	Hence, the above two displays imply 
	\[
	\|\Delta\|_{2}^{2}=\|\Delta_{N}\|_{2}^{2}+\|\Delta_{N^{c}}\|_{2}^{2}\leq6\left(4D_{1}^{-1}c_{1}^{-2}\exp(M_{1}M_{4})\right)^{2}\lambda_{\theta,}^{2}s_{\theta}.
	\]
	
	Hence, $\|\Delta\|_{2}=O_{P}(\sqrt{s_{\theta}n^{-1}\log p})$. The
	proof is complete. 
\end{proof}

\subsubsection{Combining Lasso-type and Dantzig-type estimators}
\begin{lem}
	\label{lem: overall bnd theta}Let Assumptions \ref{assu: RE condition}
	and \ref{assu: unconfoundedness} hold. Assume that $\|\theta_{(1)}\|_{0}=o(n/\log p)$.
	
	Then, $\|\hat{\theta}_{(1),A}\|_{\infty}\leq M_{5}\vee\kappa_{0}$
	and $\hat{\theta}_{(1),A}-\theta_{(1)}\in\mathcal{C}({\rm supp}(\theta_{(1)},3)$
	with probability approaching one. 
	
	Moreover, 
	$$\|\hat{\theta}_{(1),A}-\theta_{(1)}\|_{1}=O_{P}(s_{\theta}\sqrt{n^{-1}\log p}),$$
	$$\|\hat{\theta}_{(1),A}-\theta_{(1)}\|_{2}=O_{P}(\sqrt{s_{\theta}n^{-1}\log p}),$$
	$$\sum_{i\in H_{A}}W_{i}\left(q(X_{i}'\hat{\theta}_{(1),A})-q(X_{i}'\theta_{(1)})\right)^{2}=o_{P}(n)$$
	and 
	$$\EE_{n,H_{A}}(X_{i}'(\hat{\theta}_{(1),A}-\theta_{(1)}))^{2}=O_{P}(s_{\theta}n^{-1}\log p).$$
	Analogous results hold if we replace $\EE_{n,H_{A}}$ and $\hat{\theta}_{(1),A}$
	with $\EE_{n,H_{B}}$ and $\hat{\theta}_{(1),B}$. 
\end{lem}
\begin{proof}[Proof of Lemma \ref{lem: overall bnd theta}]
	By construction in Algorithm \ref{alg1}, $\|\hat{\theta}_{(1),A}\|_{1}\leq\max\{\|\tilde{\theta}_{(1),A}\|_{1},\kappa_{0}\}$.
	By Lemma \ref{lem: DS theta bnd}, $\PP(\|\tilde{\theta}_{(1),A}\|_{1}\leq M_{5})\rightarrow1$.
	Thus, $\PP(\|\hat{\theta}_{(1),A}\|_{1}\leq M_{5}\vee\kappa_{0})\rightarrow1$.
	By Hölder's inequality, $\|X\hat{\theta}_{(1),A}\|_{\infty}\leq\|X\|_{\infty}\|\hat{\theta}_{(1),A}\|_{1}\leq M_{1}(M_{5}\vee\kappa_{0})$
	with probability approaching one. Since the bounds for $\|\check{\theta}_{(1),A}-\theta_{(1)}\|_{1}$
	and $\|\tilde{\theta}_{(1),A}-\theta_{(1)}\|_{1}$ are both $O_{P}(s_{\theta}\sqrt{n^{-1}\log p})$,
	we have $\|\hat{\theta}_{(1),A}-\theta_{(1)}\|_{1}=O_{P}(s_{\theta}\sqrt{n^{-1}\log p})$.
	Similarly, $\|\hat{\theta}_{(1),A}-\theta_{(1)}\|_{2}=O_{P}(\sqrt{s_{\theta}n^{-1}\log p})$. 
	
	Now we show the last bound. Let $\Delta=\hat{\theta}_{(1),A}-\theta_{(1)}$.
	By Taylor's theorem, there exists $\tau_{i}\in[0,1]$ such that 
	\[
	\left|q(X_{i}'\hat{\theta}_{(1),A})-q(X_{i}'\theta_{(1)})\right|=\exp(-X_{i}'\theta_{(1)})\left|\exp(-X_{i}'\Delta)-1\right|=\exp(-X_{i}'\theta_{(1)})\exp(-X_{i}'\Delta\tau_{i})\left|X_{i}'\Delta\right|.
	\]
	
	Notice that 
	\[
	\|X\Delta\|_{\infty}\leq\|X\|_{\infty}\|\Delta\|_{1}\leq M_{1}(\|\theta_{(1)}\|_{1}+\|\hat{\theta}_{(1),A}\|_{1})\leq M_{1}(M_{5}+\|\hat{\theta}_{(1),A}\|_{1}).
	\]
	
	We have shown that with probability approaching one, $\|\hat{\theta}_{(1),A}\|_{1}\leq M_{5}\vee\kappa_{0}$.
	Thus, $\PP(\mathcal{A})\rightarrow1$, where the event is defined as
	$\mathcal{A}=\{\|X\Delta\|_{\infty}\leq M_{1}(M_{5}+M_{5}\vee\kappa_{0})\}$.
	Moreover, we have $\|X\theta_{(1)}\|_{\infty}\leq\|X\|_{\infty}\|\theta_{(1)}\|_{1}\leq M_{1}M_{5}$.
	It follows that with probability one, for any $i\in H_{A}$, 
	\[
	\left|q(X_{i}'\tilde{\theta}_{(1),A})-q(X_{i}'\theta_{(1)})\right|\leq\exp\left(M_{1}M_{5}+M_{1}(M_{5}+M_{5}\vee\kappa_{0})\right)|X_{i}'\Delta|.
	\]
	
	Hence, 
	\[
	\sum_{i\in H_{A}}W_{i}\left(q(X_{i}'\tilde{\theta}_{(1),A})-q(X_{i}'\theta_{(1)})\right)^{2}\leq\exp\left(2M_{1}M_{5}+2M_{1}(M_{5}+M_{5}\vee\kappa_{0})\right)\sum_{i\in H_{A}}(X_{i}'\Delta)^{2}.
	\]
	
	By Lemmas \ref{lem: lasso bnd theta-1} and \ref{lem: DS theta bnd},
	$\check{\theta}_{(1),A}-\theta_{(1)}\in\mathcal{C}({\rm supp}(\theta_{(1)}),3)$
	and $\tilde{\theta}_{(1),A}-\theta_{(1)}\in\mathcal{C}({\rm supp}(\theta_{(1)},1)$
	with probability approaching one. Since $\Delta\in\{\check{\theta}_{(1),A}-\theta_{(1)},\ \tilde{\theta}_{(1),A}-\theta_{(1)}\}$
	and $\mathcal{C}({\rm supp}(\theta_{(1)},1)\subset\mathcal{C}({\rm supp}(\theta_{(1)},3)$,
	we have that $\PP(\Delta\in\mathcal{C}({\rm supp}(\theta_{(1)}),3))\rightarrow1$. 
	
	It follows that on the event $\mathcal{M}$, 
	\[
	\EE_{n,H_{A}}(X_{i}'\Delta)^{2}\leq c_{2}^{2}\|\Delta\|_{2}^{2}=O_{P}(s_{\theta}n^{-1}\log p).
	\]
	
	By the above two displays, together with $s_{\theta}=o(n/\log p)$,
	we have $$\sum_{i\in H_{A}}W_{i}\left(q(X_{i}'\tilde{\theta}_{(1),A})-q(X_{i}'\theta_{(1)})\right)^{2}=o_{P}(n).$$
\end{proof}

\subsubsection{Estimators for the outcome model's $\beta_{(1)}$}
\begin{lem}
	\label{lem: lasso bnd beta}Let Assumptions \ref{assu: RE condition}
	and \ref{assu: unconfoundedness} hold. Also assume that $\|\beta_{(1)}\|_{0}=o(n/\log p)$.
	Then $\beta_{(1)}$ satisfies (\ref{eq: KKT beta}) and $\hat{\beta}_{(1),A}-\beta_{(1)}\in\mathcal{C}\left({\rm supp}(\beta_{(1)}),3\right)$
	with probability approaching one. 
	
	Moreover, 
	$$\|\hat{\beta}_{(1),A}-\beta_{(1)}\|_{1}=O_{P}(\|\beta_{(1)}\|_{0}\sqrt{n^{-1}\log p})$$
	and 
	$$\EE_{n,H_{B}}\left[X_{i}'(\hat{\beta}_{(1),A}-\beta_{(1)})\right]^{2}=o_{P}(1).$$
	Analogous results hold if we replace $\EE_{n,H_{A}}$ and $\hat{\beta}_{(1),A}$
	with $\EE_{n,H_{B}}$ and $\hat{\beta}_{(1),B}$. 
\end{lem}
\begin{proof}[Proof of Lemma \ref{lem: lasso bnd beta}]
	We use the standard argument for Lasso. Let $s_{\beta}=\|\beta_{(1)}\|_{0}$
	and $Q={\rm supp}(\beta_{(1)})$. Define the event 
	\begin{multline*}
	\mathcal{M}=\left\{ \|\EE_{n,H_{A}}X_{i}W_{i}\exp(X_{i}'\hat{\theta}_{(1),A})\varepsilon_{i,(1)}\|_{\infty}\leq\lambda_{\beta}/4\right\} \\
	\bigcap\left\{ \min_{|J|\leq2s_{\beta}}\min_{v\in\mathcal{C}(J,3)}\frac{\EE_{n,H_{A}}W_{i}(X_{i}'v)^{2}}{\|v_{J}\|_{2}^{2}}\geq c_{1}^{2}\right\} 
	\\
	 \bigcap\left\{ \max_{|J|\leq2s_{\beta}}\max_{v\in\mathcal{C}(J,3)}\frac{\EE_{n,H_{B}}(X_{i}'v)^{2}}{\|v\|_{2}^{2}}\leq c_{2}^{2}\right\} ,
	\end{multline*}
	where $c_{1},c_{2}>0$ are constants from Lemma \ref{lem: RE condition}.
	By Lemma \ref{lem: RE condition}, $\PP(\mathcal{M})\rightarrow1$.
	
	Define $\Delta=\hat{\beta}_{(1)}-\beta_{(1)}$ and $\tilde{W}_{i}=W_{i}\exp(-X_{i}'\hat{\theta}_{(1),A})$.
	Since (\ref{eq: KKT beta}) is the KKT condition for the optimization
	program that defineds $\hat{\beta}_{(1)}$, $\hat{\beta}_{(1)}$ satisfies
	(\ref{eq: KKT beta}). Now we show the other claims in three steps. 
	
	\textbf{Step 1:} $\PP(\Delta\in\mathcal{C}(Q,3))\rightarrow1$. 
	
	Since $W_{i}Y_{i}=W_{i}Y_{i}(1)$, we have that $\tilde{W}_{i}(Y_{i}-X_{i}'\beta)^{2}=\tilde{W}_{i}(Y_{i}(1)-X_{i}'\beta)^{2}$.
	Recall that $Y_{i}(1)-X_{i}'\beta_{(1)}=\varepsilon_{i,(1)}$. Thus,
	by construction, we have
	\[
	\EE_{n,H_{A}}\tilde{W}_{i}(\varepsilon_{i,(1)}-X_{i}'\Delta)^{2}+\lambda_{\beta}\|\beta_{(1)}+\Delta\|_{1}\leq \EE_{n,H_{A}}\tilde{W}_{i}\varepsilon_{i,(1)}^{2}+\lambda_{\beta}\|\beta_{(1)}\|_{1}.
	\]
	
	Rearranging terms, we obtain 
	\begin{align*}
	\EE_{n,H_{A}}\tilde{W}_{i}(X_{i}'\Delta)^{2} & \leq2\EE_{n,H_{A}}\varepsilon_{i,(1)}\tilde{W}_{i}X_{i}'\Delta+\lambda_{\beta}\left(\|\beta_{(1)}\|_{1}-\|\beta_{(1)}+\Delta\|_{1}\right)\\
	& =2\EE_{n,H_{A}}\varepsilon_{i,(1)}\tilde{W}_{i}X_{i}'\Delta+\lambda_{\beta}\left(\|\beta_{(1)}\|_{1}-\|\beta_{(1)}+\Delta_{Q}\|_{1}-\|\Delta_{Q^{c}}\|_{1}\right)\\
	& \leq2\EE_{n,H_{A}}\varepsilon_{i,(1)}\tilde{W}_{i}X_{i}'\Delta+\lambda_{\beta}\left(\|\Delta_{Q}\|_{1}-\|\Delta_{Q^{c}}\|_{1}\right)\\
	& \leq2\|\EE_{n,H_{A}}\varepsilon_{i,(1)}\tilde{W}_{i}X_{i}\|_{\infty}\|\Delta\|_{1}+\lambda_{\beta}\left(\|\Delta_{Q}\|_{1}-\|\Delta_{Q^{c}}\|_{1}\right).
	\end{align*}
	
	Therefore, on the event $\mathcal{M}$, we have that 
	\begin{equation}
	\EE_{n,H_{A}}\tilde{W}_{i}(X_{i}'\Delta)^{2}\leq\lambda_{\beta}\|\Delta\|_{1}/2+\lambda_{\beta}\left(\|\Delta_{Q}\|_{1}-\|\Delta_{Q^{c}}\|_{1}\right)=\frac{3}{2}\lambda_{\beta}\|\Delta_{Q}\|_{1}-\frac{1}{2}\lambda_{\beta}\|\Delta_{Q^{c}}\|_{1}.\label{eq: lasso beta 5}
	\end{equation}
	
	Hence, on the event $\mathcal{M}$, $\frac{3}{2}\lambda_{\beta}\|\Delta_{Q}\|_{1}\geq\frac{1}{2}\lambda_{\beta}\|\Delta_{Q^{c}}\|_{1}$,
	which means $\Delta\in\mathcal{C}(Q,3)$. Since $\PP(\mathcal{M})\rightarrow1$,
	we have proved $\PP(\Delta\in\mathcal{C}(Q,3))\rightarrow1$. 
	
	\textbf{Step 2:} show $\|\Delta\|_{1}=O_{P}(s_{\beta}\sqrt{n^{-1}\log p})$.
	
	By definition of $\hat{\theta}_{(1),A}$, $\|\hat{\theta}_{(1),A}\|_{1}\leq\|\theta_{(1)}\|_{1}\leq M_{5}$.
	By assumption, $\|X\|_{\infty}\leq M_{1}$. Thus, $\min_{i\in H_{A}}\exp(-X_{i}'\hat{\theta}_{(1),A})\geq\exp(-M_{1}M_{5})$.
	Since on the event $\mathcal{M}$, $\Delta\in\mathcal{C}(Q,3)$, it
	follows that on this event, 
	\[
	\EE_{n,H_{A}}\tilde{W}_{i}(X_{i}'\Delta)^{2}\geq c_{1}^{2}\exp(-M_{1}M_{5})\|\Delta_{J}\|_{2}^{2}.
	\]
	
	Hence, (\ref{eq: lasso beta 5}) implies that 
	\[
	c_{1}^{2}\exp(-M_{1}M_{5})\|\Delta_{J}\|_{2}^{2}\leq\frac{3}{2}\lambda_{\beta}\|\Delta_{Q}\|_{1}-\frac{1}{2}\lambda_{\beta}\|\Delta_{Q^{c}}\|_{1}\leq\frac{3}{2}\lambda_{\beta}\|\Delta_{Q}\|_{1}\leq\frac{3}{2}\lambda_{\beta}\|\Delta_{Q}\|_{2}\sqrt{s_{\beta}},
	\]
	which means 
	\begin{equation}
	\|\Delta_{J}\|_{2}\leq\frac{3}{2}c_{1}^{-2}\exp(M_{1}M_{5})\lambda_{\beta}\sqrt{s_{\beta}}.\label{eq: lasso beta 6}
	\end{equation}
	
	Therefore, on the event $\mathcal{M}$, 
	\[
	\|\Delta\|_{1}=\|\Delta_{Q}\|_{1}+\|\Delta_{Q^{c}}\|_{1}\leq4\|\Delta_{Q}\|_{1}\leq4\sqrt{s_{\beta}}\|\Delta_{Q}\|_{2}\leq6c_{1}^{-2}\exp(M_{1}M_{5})\lambda_{\beta}s_{\beta}.
	\]
	
	Since $\PP(\mathcal{M})\rightarrow1$ and $\lambda_{\beta}\asymp\sqrt{n^{-1}\log p}$,
	we have $\|\Delta\|_{1}=O_{P}(s_{\beta}\sqrt{n^{-1}\log p})$. 
	
	\textbf{Step 3:} show that $\EE_{n,H_{B}}\left[X_{i}'(\hat{\beta}_{(1),A}-\beta_{(1)})\right]^{2}=o_{P}(1)$. 
	
	Recall that $\Delta\in\mathcal{C}(Q,3)$ on the event $\mathcal{M}$.
	By the definition of this event, it follows that on this event, 
	\[
	\EE_{n,H_{B}}(X_{i}'\Delta)^{2}\leq c_{2}^{2}\|\Delta\|_{2}^{2}\overset{\text{(i)}}{\leq}\frac{9}{4}c_{2}^{2}c_{1}^{-4}\exp(2M_{1}M_{5})\lambda_{\beta}^{2}s_{\beta},
	\]
	where (i) follows by (\ref{eq: lasso beta 6}). By the assumption
	of $\|\beta_{(1)}\|_{0}=o(n/\log p)$, it follows that $\EE_{n,H_{B}}(X_{i}'\Delta)^{2}=o_{P}(1)$.
	The proof is complete. 
\end{proof}
\begin{lem}
	\label{lem: matrix sparse eigen bernstein}Let $\{X_{i}\}_{i=1}^{m}$
	be i.i.d sub-Gaussian random vectors. Let $\{\tilde{X}_{i}\}_{i=1}^{m}$
	be an independent copy of $\{X_{i}\}_{i=1}^{m}$. Define $\mathcal{D}(J)=\{h\in\mathbb{R}^{p}:\ {\rm support}(h)\subseteq J,\ \|h\|_{2}=1\}$,
	where $J\subseteq\{1,...,p\}$. Then
	\[
	\sup_{h_{1},h_{2}\in\mathcal{D}(J)}\left|h_{1}'\left[\frac{1}{m}\sum_{i=1}^{m}(X_{i}X_{i}'-\tilde{X}_{i}\tilde{X}_{i}')\right]h_{2}\right|=O_{P}\left(\sqrt{|J|/m}\right).
	\]
\end{lem}
\begin{proof}
	Let $s=|J|$. Let $X_{i,J}=Z_{i}$ and $\tilde{X}_{i,J}=\tilde{Z}_{i}$.
	Denote $\Sigma=\EE (Z_{i}Z_{i}')=\EE (\tilde{Z}_{i}\tilde{Z}_{i}')$. By
	the sub-Gaussian assumption, it follows by Proposition 2.1 of \citet{vershynin2012close}
	that 
	\[
	\left\Vert \frac{1}{m}\sum_{i=1}^{m}(Z_{i}Z_{i}'-\Sigma)\right\Vert _{{\rm spectral}}=O_{P}\left(\sqrt{|J|/m}\right),
	\]
	where $\|\cdot\|_{{\rm spectral}}$ denotes the spectral norm, i.e.,
	the maximal singular value. Similarly, we can show 
	\[
	\left\Vert \frac{1}{m}\sum_{i=1}^{m}(\tilde{Z}_{i}\tilde{Z}_{i}'-\Sigma)\right\Vert _{{\rm spectral}}=O_{P}\left(\sqrt{|J|/m}\right).
	\]
	
	Thus, 
	\[
	\left\Vert \frac{1}{m}\sum_{i=1}^{m}(Z_{i}Z_{i}'-\tilde{Z}_{i}\tilde{Z}_{i}')\right\Vert _{{\rm spectral}}=O_{P}\left(\sqrt{|J|/m}\right).
	\]
	
	The desired result follows by notice that 
	\[
	\left|h_{1}'\left[\frac{1}{m}\sum_{i=1}^{m}(X_{i}X_{i}'-\tilde{X}_{i}\tilde{X}_{i}')\right]h_{2}\right|\leq\|h_{1}\|_{2}\|h_{2}\|_{2}\left\Vert \frac{1}{m}\sum_{i=1}^{m}(Z_{i}Z_{i}'-\tilde{Z}_{i}\tilde{Z}_{i}')\right\Vert _{{\rm spectral}}.
	\]
\end{proof}
\begin{lem}
	\label{lem: matrix restricted eigen bernstein}Let $\{X_{i}\}_{i=1}^{m}$
	be i.i.d sub-Gaussian random vectors . Let $\{\tilde{X}_{i}\}_{i=1}^{m}$
	be an independent copy of $\{X_{i}\}_{i=1}^{m}$. Let $J\subset\{1,...,p\}$.
	Then 
	\[
	\sup_{h_{1},h_{2}\in\mathcal{C}(J,3)\bigcap\mathbb{S}^{p-1}}\left|h_{1}'\left[\frac{1}{m}\sum_{i=1}^{m}(X_{i}X_{i}'-\tilde{X}_{i}\tilde{X}_{i}')\right]h_{2}\right|=O_{P}\left(\sqrt{|J|/m}\right).
	\]

\end{lem}
\begin{proof}[Proof of Lemma \ref{lem: matrix restricted eigen bernstein}]
	Let $s=|J|$. Fix $\delta\in(0,1)$. By Lemma 14 of \citet{rudelson2013reconstruction},
	we have 
	\[
	\bigcup_{|J|\leq s}\mathcal{C}(J,3)\bigcap\mathbb{S}^{p-1}\subset(1-\delta)^{-1}{\rm conv}\left(\bigcup_{|J|\leq d}V_{J}\bigcap\mathbb{S}^{p-1}\right)
	\]
	for some $d\asymp s$, where $V_{J}$ is the space spanned by $\{e_{j}\}_{j\in J}$
	and $e_{j}$ is the $j$-th column of the $p\times p$ identity matrix.
	
	Let $\Omega=\frac{1}{m}\sum_{i=1}^{m}(X_{i}X_{i}'-\tilde{X}_{i}\tilde{X}_{i}')$.
	We first fix $h_{2}\in\mathcal{C}(J)\bigcap\mathbb{S}^{p-1}$. Notice
	that
	\begin{align}
	\sup_{h_{1}\in\mathcal{C}(J)\bigcap\mathbb{S}^{p-1}}\left|h_{1}'\Omega h_{2}\right|
	&
	\leq\frac{1}{1-\delta}\sup_{h_{1}\in{\rm conv}\left(\bigcup_{|J|\leq d}V_{J}\bigcap\mathbb{S}^{p-1}\right)}\left|h_{1}'\Omega h_{2}\right|
	\\
	&\overset{{\rm (i)}}{=}\frac{1}{1-\delta}\sup_{h_{1}\in\bigcup_{|J|\leq d}V_{J}\bigcap\mathbb{S}^{p-1}}\left|h_{1}'\Omega h_{2}\right|,\label{eq: matrix RE bernstein}
	\end{align}
	where (i) follows by the fact that the maximum takes place at extreme
	points of the set ${\rm conv}\left(\bigcup_{|J|\leq d}V_{J}\bigcap\mathbb{S}^{p-1}\right)$
	due to the convexity of the mapping $h_{1}\mapsto\left|h_{1}'\Omega h_{2}\right|$.
	
	Define the function $h_{2}\mapsto f(h_{2})$ by 
	$$f(h_{2})=\sup_{h_{1}\in\mathcal{C}(J)\bigcap\mathbb{S}^{p-1}}\left|h_{1}'\Omega h_{2}\right|.$$
	Notice that $f(\cdot)$ is convex since it is the supreme of convex
	functions. We observe that 
	\begin{align*}
	\sup_{h_{1},h_{2}\in\mathcal{C}(J)\bigcap\mathbb{S}^{p-1}}\left|h_{1}'\Omega h_{2}\right| & =\sup_{h_{2}\in\mathcal{C}(J)\bigcap\mathbb{S}^{p-1}}f(h_{2})
	\\
	&\overset{\text{(i)}}{\leq}\frac{1}{1-\delta}\sup_{h_{2}\in{\rm conv}\left(\bigcup_{|J|\leq d}V_{J}\bigcap\mathbb{S}^{p-1}\right)}f(h_{2})\\
	&\overset{{\rm (ii)}}{=}\frac{1}{1-\delta}\sup_{h_{2}\in\bigcup_{|J|\leq d}V_{J}\bigcap\mathbb{S}^{p-1}}f(h_{2}),
	\end{align*}

	where (i) follows by (\ref{eq: matrix RE bernstein}) and (ii) follows
	by the fact that the maximum takes place at extreme points of the
	set ${\rm conv}\left(\bigcup_{|J|\leq d}V_{J}\bigcap\mathbb{S}^{p-1}\right)$
	due to the convexity of $f(\cdot)$. Therefore, we have 
	\[
	\sup_{h_{1},h_{2}\in\mathcal{C}(J)\bigcap\mathbb{S}^{p-1}}\left|h_{1}'\Omega h_{2}\right|\leq\frac{1}{1-\delta}\sup_{h_{1},h_{2}\in\bigcup_{|J|\leq d}V_{J}\bigcap\mathbb{S}^{p-1}}\left|h_{1}'\Omega h_{2}\right|=\frac{1}{1-\delta}\sup_{h_{1},h_{2}\in\mathcal{D}(J)}\left|h_{1}'\Omega h_{2}\right|.
	\]
	
	Thus, the desired result follows by Lemma \ref{lem: matrix sparse eigen bernstein}.
\end{proof}

\subsection{Proof of Theorem \ref{thm: sparse beta}}\label{sec: proof thm sparse beta}
\begin{proof}[\textbf{Proof of Theorem \ref{thm: sparse beta}}]
	We observe the following decomposition
	\[
	2b_{n}(\hat{\mu}_{(1)}-\mu_{(1)})=Q_{A}+Q_{B},
	\]
	where 
	\[
	Q_{A}=\sum_{i\in H_{A}}\left[W_{i}Y_{i}q(X_{i}'\hat{\theta}_{(1),A})+\left(1-W_{i}q(X_{i}'\hat{\theta}_{(1),A})\right)X_{i}'\hat{\beta}_{(1),B}-\mu_{(1)}\right]
	\]
	and 
	\[
	Q_{B}=\sum_{i\in H_{B}}\left[W_{i}Y_{i}q(X_{i}'\hat{\theta}_{(1),B})+\left(1-W_{i}q(X_{i}'\hat{\theta}_{(1),B})\right)X_{i}'\hat{\beta}_{(1),A}-\mu_{(1)}\right].
	\]
	
	We now characterize $Q_{A}$; completely analogous arguments hold
	for $Q_{B}$. Define 
	\[
	D_{1,n}=\sum_{i\in H_{A}}\left[1-W_{i}q(X_{i}'\hat{\theta}_{(1),A})\right]X_{i}'\left(\hat{\beta}_{(1),B}-\beta_{(1)}\right)
	\]
	and 
	\[
	D_{2,n}=\sum_{i\in H_{A}}W_{i}\varepsilon_{i,(1)}\left(q(X_{i}'\hat{\theta}_{(1),A})-q(X_{i}'\theta_{(1)})\right).
	\]
	
	Notice that
	\begin{align}
	Q_{A} & \overset{{\rm (i)}}{=}\sum_{i\in H_{A}}\left[W_{i}\varepsilon_{i,(1)}q(X_{i}'\hat{\theta}_{(1),A})+\left[1-W_{i}q(X_{i}'\hat{\theta}_{(1),A})\right]X_{i}'\left(\hat{\beta}_{(1),B}-\beta_{(1)}\right)+X_{i}'\beta_{(1)}-\mu_{(1)}\right]\nonumber \\
	& =\sum_{i\in H_{A}}\left[W_{i}\varepsilon_{i,(1)}q(X_{i}'\theta_{(1)})+X_{i}'\beta_{(1)}-\mu_{(1)}\right]+D_{1,n}+D_{2,n},\label{eq: sparse beta eq 1}
	\end{align}
	where (i) follows by $W_{i}Y_{i}=W_{i}Y_{i}(1)=W_{i}(X_{i}'\beta_{(1)}+\varepsilon_{i,(1)})$. 
	
	Now we bound $D_{1,n}$ and $D_{2,n}$. Using Hölder's inequality,
	we have 
	\begin{align}
	\left|D_{1,n}\right|&\leq\left\Vert \sum_{i\in H_{A}}\left[1-W_{i}q(X_{i}'\hat{\theta}_{(1),A})\right]X_{i}\right\Vert _{\infty}\|\hat{\beta}_{(1),B}-\beta_{(1)}\|_{1}
	\\
	&\overset{{\rm (i)}}{\leq}b_{n}\lambda_{\theta}\|\hat{\beta}_{(1),B}-\beta_{(1)}\|_{1}\\
	&\overset{{\rm (ii)}}{=}b_{n}O\left(\sqrt{n^{-1}\log p}\right)O_{P}\left(\|\beta_{(1)}\|_{0}\sqrt{b_{n}^{-1}\log p}\right)\overset{{\rm (iii)}}{=}o_{P}(\sqrt{n}),\label{eq: sparse beta eq 1.5}
	\end{align}
	where (i) follows by (\ref{eq: KKT theta}), (ii) follows by $b_{n}\asymp n$,
	$\lambda_{\theta}=O(\sqrt{n^{-1}\log p})$ and Lemma \ref{lem: lasso bnd beta}
	and (iii) follows by $\|\beta_{(1)}\|_{0}=o(\sqrt{n}/\log p)$. 
	
	Notice that $\{\varepsilon_{i,(1)}\}_{i\in H_{A}}$ is independent
	across $i$ conditional on $\{(X_{i},W_{i})\}_{i\in H_{A}}$ and that
	$\EE(\varepsilon_{i,(1)}\mid\{(X_{i},W_{i})\}_{i\in H_{A}})=0$. Therefore,
	\begin{align*}
	&\EE \left[D_{2,n}^{2}\mid\{(X_{i},W_{i})\}_{i\in H_{A}}\right] 
	\\
	&\qquad  =\sum_{i\in H_{A}}\EE(\varepsilon_{i,(1)}^{2}\mid\{(X_{i},W_{i})\}_{i\in H_{A}})W_{i}\left(q(X_{i}'\hat{\theta}_{(1),A})-q(X_{i}'\theta_{(1)})\right)^{2}\\
	&\qquad   =\sum_{i\in H_{A}}\EE(\varepsilon_{i,(1)}^{2}\mid X_{i})\left(q(X_{i}'\hat{\theta}_{(1),A})-q(X_{i}'\theta_{(1)})\right)^{2}W_{i}\\
	&\qquad   \leq\left(\max_{1\leq i\leq n}\EE(\varepsilon_{i,(1)}^{2}\mid X_{i})\right)\left[\sum_{i\in H_{A}}W_{i}\left(q(X_{i}'\hat{\theta}_{(1),A})-q(X_{i}'\theta_{(1)})\right)^{2}\right]\overset{{\rm (i)}}{=}o_{P}(n),
	\end{align*}
	where (i) follows by $\max_{1\leq i\leq n}\EE(\varepsilon_{i,(1)}^{2}\mid X_{i})=O_{P}(1)$
	(due to the assumption of sub-Gaussian $\varepsilon_{i,(1)}$) and
	Lemma \ref{lem: overall bnd theta}. Hence, 
	\begin{equation}
	D_{2,n}=o_{P}(\sqrt{n}).\label{eq: sparse beta eq 2}
	\end{equation}
	
	In light of the decomposition in (\ref{eq: sparse beta eq 1}), it
	follows from (\ref{eq: sparse beta eq 1.5}) and (\ref{eq: sparse beta eq 2})
	that 
	\[
	Q_{A}=\sum_{i\in H_{A}}\left[W_{i}\varepsilon_{i,(1)}q(X_{i}'\theta_{(1)})+X_{i}'\beta_{(1)}-\mu_{(1)}\right]+o_{P}(\sqrt{n}).
	\]
	
	Similarly, we can show that 
	\[
	Q_{B}=\sum_{i\in H_{B}}\left[W_{i}\varepsilon_{i,(1)}q(X_{i}'\theta_{(1)})+X_{i}'\beta_{(1)}-\mu_{(1)}\right]+o_{P}(\sqrt{n}).
	\]
	
	The desired result follows. 
\end{proof}

\subsection{Proof of Theorem \ref{thm: sparse theta}}\label{sec: proof thm sparse theta}

\subsubsection{Preliminary results for proving Theorem \ref{thm: sparse theta}}
\begin{lem}
	\label{lem: RE cond higer order}Let Assumption \ref{assu: RE condition}
	hold. Suppose that $s\ll\sqrt{n}/\log p$. Then there exists a constant
	$c_{3}>0$ depending only on $M_{1}$ such that
	\[
	\PP\left(\max_{v\in\bigcup_{|J|\leq s}\mathcal{C}(J,1)}\frac{\left[\EE_{n,H_{A}}(X_{i}'v)^{4}\right]^{1/4}}{\|v\|_{2}}\leq c_{3}\right)\rightarrow1.
	\]
\end{lem}
\begin{proof}[Proof of Lemma \ref{lem: RE cond higer order}]
	The
	proof proceeds in two steps. We first derive a large deviation bound
	for $\EE_{n,H_{A}}(X_{i}'v)^{4}-\EE(X_{i}'v)^{4}$ and then use a covering
	argument and reduction principle to prove the desired result. 
	
	\textbf{Step 1:} bound $\EE_{n,H_{A}}(X_{i}'v)^{4}-\EE(X_{i}'v)^{4}$
	for any $v\in\mathbb{S}^{p-1}$.
	
	Fix $v\in\mathbb{S}^{p-1}$. Let $Z_{i}=(X_{i}'v)^{4}-\EE(X_{i}'v)^{4}$.
	Notice that $X_{i}'v$ has a bounded sub-Gaussian norm by Assumption
	\ref{assu: RE condition}. It follows that $|Z_{i}|^{1/4}$ is sub-Gaussian.
	Hence, there exists a constant $C_{1}>0$ depending only on $M_{1}$
	such that 
	$$\PP(|Z_{i}|^{1/4}>z)\leq\exp(1-C_{1}z^{2})$$ for any $z>0.$
	Thus, $\PP(|Z_{i}|>z)\leq\exp(1-C_{1}z^{1/2})$ for $z>0$. Now we apply
	Theorem 1 of \citet{merlevede2011bernstein}. Notice that since we
	have i.i.d data, we can take $\gamma_{1}$ in Equation (2.6) therein
	to be $\infty$ and thus we can take $\gamma=1/2$ in their notation.
	It follows by their Theorem 1 and Remark 3 that there exist constants
	$C_{2},...,C_{6}>0$ depending only on $M_{1}$ such that for any
	$t>0$, 
	\begin{multline}
	\PP\left(\left|\sum_{i=1}^{b_{n}}\left((X_{i}'v)^{4}-\EE (X_{i}'v)^{4}\right)\right|\ge t\right)	=\PP\left(\left|\sum_{i=1}^{b_{n}}Z_{i}\right|\ge t\right)\\
\leq b_{n}\exp\left(-C_{2}t^{1/2}\right)+\exp\left(-\frac{t^{2}}{C_{3}+C_{4}b_{n}}\right)+\exp\left[-C_{5}b_{n}^{-1}t^{2}\exp\left(C_{6}t^{1/4}/\sqrt{\log t}\right)\right].\label{eq: RE higher order 3}
	\end{multline}
	
	Moreover, the bounded sub-Gaussian norm of $X_{i}'v$ implies that
	there exists a constant $C_{7}>0$ depending on $M_{1}$ such that
	\begin{equation}
	\EE(X_{i}'v)^{4}\leq C_{7}.\label{eq: RE higher order 3.5}
	\end{equation}
	
	\textbf{Step 2:} prove the desired result.
	
	Fix $\delta\in(0,1)$. Let $X_{A}=(X_{1},...,X_{b_{n}})'\in\mathbb{R}^{b_{n}\times p}$.
	By Lemma 14 of \citet{rudelson2013reconstruction}, we have 
	\[
	\bigcup_{|J|\leq s}\mathcal{C}(J,1)\bigcap\mathbb{S}^{p-1}\subset(1-\delta)^{-1}{\rm conv}\left(\bigcup_{|J|\leq d}V_{J}\bigcap\mathbb{S}^{p-1}\right)
	\]
	for $d=C_{8}s$, where $C_{8}>0$ is a constant depending only on
	$\delta$, $V_{J}$ is the space spanned by $\{e_{j}\}_{j\in J}$
	and $e_{j}$ is the $j$-th column of the $p\times p$ identity matrix.
	Notice that
	\begin{align}
	\sup_{v\in\bigcup_{|J|\leq s}\mathcal{C}(J,1)\bigcap\mathbb{S}^{p-1}}\|X_{A}v\|_{4}&\leq\frac{1}{1-\delta}\sup_{v\in{\rm conv}\left(\bigcup_{|J|\leq d}V_{J}\bigcap\mathbb{S}^{p-1}\right)}\|X_{A}v\|_{4}\\
	&\overset{{\rm (i)}}{=}\frac{1}{1-\delta}\sup_{v\in\bigcup_{|J|\leq d}V_{J}\bigcap\mathbb{S}^{p-1}}\|X_{A}v\|_{4},\label{eq: basic RE 4}
	\end{align}
	where (i) follows by the fact that the maximum takes place at extreme
	points of the set ${\rm conv}\left(\bigcup_{|J|\leq d}V_{J}\bigcap\mathbb{S}^{p-1}\right)$
	due to the convexity of the mapping $v\mapsto\|X_{A}v\|_{4}$. 
	
	Now we use the standard covering argument. For any $J\subset\{1,...,p\}$
	with $|J|=s$, we can find a $\delta$-net 
	$$\mathcal{T}_{J}=\{v_{(1)}^{J},...,v_{(N)}^{J}\}$$
	for $V_{J}\bigcap\mathbb{S}^{p-1}$. By Lemma 20 of \citet{rudelson2013reconstruction},
	this can be done with $N\leq(3/\delta)^{s}$. Thus, we can use $\mathcal{T}=\bigcup_{|J|=d}\mathcal{T}_{J}$
	as a $\delta$-net for $\bigcup_{|J|\leq d}V_{J}\bigcap\mathbb{S}^{p-1}$.
	Notice that 
	\begin{equation}
	|\mathcal{T}|\leq(3/\delta)^{s}{p \choose d}\leq\left(\frac{3ep}{d\delta}\right)^{d}<(3e\delta^{-1}p)^{d}\label{eq: basic RE 5}
	\end{equation}
	
	Let 
	$$S=\sup_{v\in\bigcup_{|J|\leq d}V_{J}\bigcap\mathbb{S}^{p-1}}\|X_{A}v\|_{4}.$$
	For any $v_{0}\in\bigcup_{|J|\leq d}V_{J}\bigcap\mathbb{S}^{p-1}$,
	we can find $J$ with $|J|\leq s$ and $v_{1}\in\mathcal{T}$ such
	that $\|v_{1}-v_{0}\|_{2}\leq\delta$ and $v_{1},v_{0}\in V_{J}\bigcap\mathbb{S}^{p-1}$.
	Therefore, $(v_{1}-v_{0})/\|v_{1}-v_{0}\|_{2}\in V_{J}\bigcap\mathbb{S}^{p-1}$.
	Now we observe that 
	\begin{align*}
	\|X_{A}v_{0}\|_{4} & \leq\|X_{A}v_{1}\|_{4}+\|X_{A}(v_{1}-v_{0})\|_{4}\\
	& \leq\max_{v\in\mathcal{T}}\|X_{A}v\|_{4}+\|v_{1}-v_{0}\|_{2}\cdot\left\Vert X_{A}\frac{v_{1}-v_{0}}{\|v_{1}-v_{0}\|_{2}}\right\Vert _{4}\\
	& \leq\max_{v\in\mathcal{T}}\|X_{A}v\|_{4}+\delta\left\Vert X_{A}\frac{v_{1}-v_{0}}{\|v_{1}-v_{0}\|_{2}}\right\Vert _{4}\overset{\text{(i)}}{\leq}\max_{v\in\mathcal{T}}\|X_{A}v\|_{4}+\delta S,
	\end{align*}
	where (i) follows by $(v_{1}-v_{0})/\|v_{1}-v_{0}\|_{2}\in V_{J}\bigcap\mathbb{S}^{p-1}$.
	Since the above holds for any $v_{0}\in\bigcup_{|J|\leq d}V_{J}\bigcap\mathbb{S}^{p-1}$,
	we have that 
	\[
	S=\sup_{v_{0}\in\bigcup_{|J|\leq d}V_{J}\bigcap\mathbb{S}^{p-1}}\|X_{A}v_{0}\|_{4}\leq\max_{v\in\mathcal{T}}\|X_{A}v\|_{4}+\delta S,
	\]
	which means that 
	\begin{equation}
	S\leq\frac{1}{1-\delta}\max_{v\in\mathcal{T}}\|X_{A}v\|_{4}.\label{eq: basic RE 6}
	\end{equation}
	
	By (\ref{eq: basic RE 4}) and (\ref{eq: basic RE 6}), we have 
	\begin{equation}
	\sup_{v\in\bigcup_{|J|\leq s}\mathcal{C}(J,1)\bigcap\mathbb{S}^{p-1}}\|X_{A}v\|_{4}\leq(1-\delta)^{-2}\max_{v\in\mathcal{T}}\|X_{A}v\|_{4}.\label{eq: basic RE 7}
	\end{equation}
	
	By (\ref{eq: RE higher order 3}), it follows that for any $v\in\mathcal{T}$
	and for any $x>0$
	\begin{multline*}
	\PP\left(\left|\sum_{i=1}^{b_{n}}\left((X_{i}'v)^{4}-\EE(X_{i}'v)^{4}\right)\right|\ge b_{n}x\right)\\
	\leq b_{n}\exp\left(-C_{2}\sqrt{b_{n}x}\right)+\exp\left(-\frac{b_{n}^{2}x^{2}}{C_{3}+C_{4}b_{n}}\right)+\exp\left[-C_{5}b_{n}x^{2}\exp\left(C_{6}\frac{(b_{n}x)^{1/4}}{\sqrt{\log(b_{n}x)}}\right)\right].
	\end{multline*}
	
	By the union bound and (\ref{eq: basic RE 5}), it follows that for
	any $x>0$, 
	\begin{align*}
	& \PP\left(\max_{v\in\mathcal{T}}\left|\EE_{n,H_{A}}(X_{i}'v)^{4}-\EE(X_{i}'v)^{4}\right|\ge x\right)\\
	& =\PP\left(\max_{v\in\mathcal{T}}\left|\sum_{i=1}^{b_{n}}\left((X_{i}'v)^{4}-\EE(X_{i}'v)^{4}\right)\right|\ge b_{n}x\right)\\
	&\qquad  \leq\exp\left(\log b_{n}+d\log(3e\delta^{-1})+d\log p-C_{2}\sqrt{b_{n}x}\right)
	\\ 
	&\quad  \qquad +\exp\left(d\log(3e\delta^{-1})+d\log p-\frac{b_{n}^{2}x^{2}}{C_{3}+C_{4}b_{n}}\right)\\
	&\qquad  \qquad+\exp\left[d\log(3e\delta^{-1})+d\log p-C_{5}b_{n}x^{2}\exp\left(C_{6}\frac{(b_{n}x)^{1/4}}{\sqrt{\log(b_{n}x)}}\right)\right].
	\end{align*}
	
	Since $d\asymp s$ and $b_{n}\asymp n$, the assumption of $s\ll\sqrt{n}/\log p$
	implies that each of the three terms on the right-hand size of the
	above display tends to zero for any choice of $x>0$. Hence, $\max_{v\in\mathcal{T}}\left|\EE_{n,H_{A}}(X_{i}'v)^{4}-E(X_{i}'v)^{4}\right|=o_{P}(1)$.
	By (\ref{eq: RE higher order 3.5}), we have 
	\[
	\PP\left(\max_{v\in\mathcal{T}}\left|\EE_{n,H_{A}}(X_{i}'v)^{4}\right|\leq2C_{7}\right)\rightarrow1.
	\]
	
	Therefore, it follows by (\ref{eq: basic RE 7}) that 
	\[
	\PP\left(\sup_{v\in\bigcup_{|J|\leq s}\mathcal{C}(J,1)\bigcap\mathbb{S}^{p-1}}\|X_{A}v\|_{4}\leq2C_{7}/(1-\delta)^{2}\right)\rightarrow1.
	\]
	
	The proof is complete. 
\end{proof}

\begin{lem}
	\label{lem: debias sparse theta new part 1}Under the assumptions
	of Theorem \ref{thm: sparse theta}, we have 
	\[
	\sum_{i\in H_{A}}X_{i}'\delta_{\theta,A}W_{i}\dot{q}(X_{i}'\hat{\theta}_{(1),A})X_{i}'\left(\hat{\beta}_{(1),B}-\hat{\beta}_{(1)}\right)=o_{P}(\sqrt{n}).
	\]
\end{lem}
\begin{proof}[Proof of Lemma \ref{lem: debias sparse theta new part 1}]
	Let $\delta_{\theta,A}=\hat{\theta}_{(1),A}-\theta_{(1)}$ and $\delta_{\beta,B}=\hat{\beta}_{(1),B}-\hat{\beta}_{(1)}$.
	Let $\tilde{X}_{i}=X_{b_{n}+i}$, $\tilde{W}_{i}=W_{b_{n}+i}$ and
	$\tilde{Y}_{i}=Y_{b_{n}+i}$. Then we need to show 
	\[
	\sum_{i=1}^{b_{n}}W_{i}\dot{q}(X_{i}'\hat{\theta}_{(1),A})\delta_{\theta,A}'X_{i}X_{i}'\delta_{\beta,B}=o_{P}(\sqrt{n}).
	\]
	
	We decompose 
	\begin{align}
	&\sum_{i=1}^{b_{n}}W_{i}\dot{q}(X_{i}'\hat{\theta}_{(1),A})\delta_{\theta,A}'X_{i}X_{i}'\delta_{\beta,B} 
	\nonumber\\
	&  \quad=\sum_{i=1}^{b_{n}}\tilde{X}_{i}'\delta_{\theta,A}\tilde{W}_{i}\dot{q}(\tilde{X}_{i}'\hat{\theta}_{(1),B})\tilde{X}_{i}'\delta_{\beta,B}\nonumber \\
	& \qquad+\sum_{i=1}^{b_{n}}\left[\dot{q}(\tilde{X}_{i}'\theta_{(1)})-\dot{q}(\tilde{X}_{i}'\hat{\theta}_{(1),B})\right]\tilde{W}_{i}\delta_{\theta,A}'\tilde{X}_{i}\tilde{X}_{i}'\delta_{\beta,B}\nonumber \\
	& \qquad+\sum_{i=1}^{b_{n}}\left[\dot{q}(X_{i}'\theta_{(1)})W_{i}\delta_{\theta,A}'X_{i}X_{i}'\delta_{\beta,B}-\dot{q}(\tilde{X}_{i}'\theta_{(1)})\tilde{W}_{i}\delta_{\theta,A}'\tilde{X}_{i}\tilde{X}_{i}'\delta_{\beta,B}\right]\nonumber \\
	& \qquad+\sum_{i=1}^{b_{n}}\left[\dot{q}(X_{i}'\hat{\theta}_{(1),A})-\dot{q}(X_{i}'\theta_{(1)})\right]W_{i}\delta_{\theta,A}'X_{i}X_{i}'\delta_{\beta,B}.\label{eq: debias sparse theta 1}
	\end{align}
	
	We bound these four terms in four steps.
	
	\textbf{Step 1:} show that $\sum_{i=1}^{b_{n}}\tilde{X}_{i}'\delta_{\theta,A}\tilde{W}_{i}\dot{q}(\tilde{X}_{i}'\hat{\theta}_{(1),B})\tilde{X}_{i}'\delta_{\beta,B}=o_{P}(\sqrt{n})$.
	
	By the KKT condition (\ref{eq: KKT beta}), we have 
	\[
	\left\Vert \frac{1}{b_{n}}\sum_{i=1}^{b_{n}}\tilde{W}_{i}\dot{q}(\tilde{X}_{i}'\hat{\theta}_{(1),B})(\tilde{X}_{i}'\hat{\beta}_{(1),B}-\tilde{Y}_{i})\tilde{X}_{i}\right\Vert _{\infty}\leq\lambda_{\beta}/4.
	\]
	
	Notice that 
	$$\frac{1}{b_{n}}\sum_{i=1}^{b_{n}}\tilde{W}_{i}\dot{q}(\tilde{X}_{i}'\hat{\theta}_{(1),B})(\tilde{X}_{i}'\beta_{(1)}-\tilde{Y}_{i})\tilde{X}_{i}=-\EE_{n,H_{B}}W_{i}X_{i}\dot{q}(X_{i}'\hat{\theta}_{(1),B})\varepsilon_{i,(1)}.$$
	By Lemma \ref{lem: RE condition} (the version with $\EE_{n,H_{B}}$),
	we have 
	\[
	\left\Vert \frac{1}{b_{n}}\sum_{i=1}^{b_{n}}\tilde{W}_{i}\dot{q}(\tilde{X}_{i}'\hat{\theta}_{(1),B})(\tilde{X}_{i}'\beta_{(1)}-\tilde{Y}_{i})\tilde{X}_{i}\right\Vert _{\infty}\leq\lambda_{\beta}/4
	\]
	
	Hence, 
	\[
	\left\Vert \sum_{i=1}^{b_{n}}\tilde{W}_{i}\dot{q}(\tilde{X}_{i}'\hat{\theta}_{(1),B})\tilde{X}_{i}\tilde{X}_{i}'\delta_{\beta,B}\right\Vert _{\infty}=O_{P}(b_{n}\lambda_{\beta})
	\]
	
	It follows by Lemma \ref{lem: overall bnd theta} that 
	\begin{align*}
	& \left|\sum_{i=1}^{b_{n}}\tilde{W}_{i}\dot{q}(\tilde{X}_{i}'\hat{\theta}_{(1),B})\delta_{\theta,A}'\tilde{X}_{i}\tilde{X}_{i}'\delta_{\beta,B}\right|\\
	& \qquad \leq\|\delta_{\theta,A}\|_{1}\left\Vert \sum_{i=1}^{b_{n}}\tilde{W}_{i}\dot{q}(\tilde{X}_{i}'\hat{\theta}_{(1),B})\tilde{X}_{i}\tilde{X}_{i}'\delta_{\beta,B}\right\Vert _{\infty}
	\\
	& \qquad =O_{P}(b_{n}\lambda_{\theta})O_{P}\left(\|\theta_{(1)}\|_{0}\sqrt{b_{n}^{-1}\log p}\right)=o_{P}(\sqrt{n}).
	\end{align*}
	
	\textbf{Step 2:} show that $\sum_{i=1}^{b_{n}}\left[\dot{q}(\tilde{X}_{i}'\hat{\theta}_{(1),B})-\dot{q}(\tilde{X}_{i}'\theta_{(1)})\right]\tilde{W}_{i}\delta_{\theta,A}'\tilde{X}_{i}\tilde{X}_{i}'\delta_{\beta,B}=o_{P}(\sqrt{n})$.
	
	Lemma \ref{thm: sparse theta} implies $\PP(\|\hat{\theta}_{(1),B}\|_{\infty}\leq M_{5})\rightarrow1$.
	Since 
	$$\|\tilde{X}\hat{\theta}_{(1),A}\|_{\infty}\leq\|\tilde{X}\|_{\infty}\|\hat{\theta}_{(1),A}\|_{1}\leq M_{1}\|\hat{\theta}_{(1),A}\|_{1},$$
	we have $\PP(\|\tilde{X}\hat{\theta}_{(1),B}\|_{\infty}\leq M_{1}M_{5})\rightarrow1$.
	By assumption, $\|\tilde{X}\theta_{(1)}\|_{\infty}\leq\|\tilde{X}\|_{\infty}\|\theta_{(1)}\|_{1}\leq M_{1}M_{5}$.
	By Taylor's theorem, we have that 
	\[
	\max_{1\leq i\leq b_{n}}\frac{|\dot{q}(\tilde{X}_{i}'\hat{\theta}_{(1),B})-\dot{q}(\tilde{X}_{i}'\theta_{(1)})|}{|\tilde{X}_{i}'\delta_{\theta,B}|}\leq\max_{|t|\leq\max\{\|\tilde{X}\theta_{(1)}\|_{\infty},\|\tilde{X}\hat{\theta}_{(1),B}\|_{\infty}\}}|\ddot{q}(t)|=O_{P}(1).
	\]
	
	Therefore, we have 
	\begin{align*}
	& \left|\sum_{i=1}^{b_{n}}\left[\dot{q}(\tilde{X}_{i}'\hat{\theta}_{(1),B})-\dot{q}(\tilde{X}_{i}'\theta_{(1)})\right]\tilde{W}_{i}\delta_{\theta,A}'\tilde{X}_{i}\tilde{X}_{i}'\delta_{\beta,B}\right|\\
	&\qquad  \leq O_{P}(1)\sum_{i=1}^{b_{n}}|\tilde{X}_{i}'\delta_{\theta,B}|\cdot|\tilde{X}_{i}'\delta_{\theta,A}|\cdot|\tilde{X}_{i}'\delta_{\beta,B}|\\
	& \qquad\leq O_{P}(1)\sqrt{\left(\sum_{i=1}^{b_{n}}|\tilde{X}_{i}'\delta_{\theta,B}|^{2}\cdot|\tilde{X}_{i}'\delta_{\theta,A}|^{2}\right)\left(\sum_{i=1}^{b_{n}}(\tilde{X}_{i}'\delta_{\beta,B})^{2}\right)}\\
	& \qquad\leq O_{P}(1)\sqrt{\sqrt{\left(\sum_{i=1}^{b_{n}}|\tilde{X}_{i}'\delta_{\theta,B}|^{4}\right)\left(\sum_{i=1}^{b_{n}}|\tilde{X}_{i}'\delta_{\theta,A}|^{4}\right)}\left(\sum_{i=1}^{b_{n}}(\tilde{X}_{i}'\delta_{\beta,B})^{2}\right)}\\
	& \qquad\overset{{\rm (i)}}{\leq}O_{P}(1)\sqrt{\sqrt{O_{P}\left(b_{n}\|\delta_{\theta,B}\|_{2}^{4}\right)O_{P}\left(b_{n}\|\delta_{\theta,A}\|_{2}^{4}\right)}\left(\sum_{i=1}^{b_{n}}(\tilde{X}_{i}'\delta_{\beta,B})^{2}\right)}\\
	& \qquad\overset{{\rm (ii)}}{\leq}O_{P}(1)\sqrt{\sqrt{O_{P}\left(b_{n}[s_{\theta}b_{n}^{-1}\log p]^{2}\right)O_{P}\left(b_{n}[s_{\theta}b_{n}^{-1}\log p]^{2}\right)}\left(s_{\beta}\log p\right)}\overset{{\rm (iii)}}{=}o_{P}(\sqrt{n})
	\end{align*}
	where (i) follows by Lemma \ref{lem: RE cond higer order}, together
	with $\|\theta_{(1)}\|_{0}\ll\sqrt{n}/\log p$ and the fact that $\PP(\delta_{\theta,A},\delta_{\theta,B}\in\mathcal{C}({\rm supp}(\theta_{(1)}),1))\rightarrow1$
	(Lemma \ref{lem: overall bnd theta}), (ii) follows by Lemmas \ref{lem: overall bnd theta}
	and \ref{lem: lasso bnd beta} and (iii) follows by $b_{n}\asymp n$,
	$s_{\theta}\ll\sqrt{n}/\log p$ and $s_{\beta}\ll n/\log p$. 
	
	\textbf{Step 3:} show that $\sum_{i=1}^{b_{n}}\left[\dot{q}(X_{i}'\theta_{(1)})W_{i}\delta_{\theta,A}'X_{i}X_{i}'\delta_{\beta,B}-\dot{q}(\tilde{X}_{i}'\theta_{(1)})\tilde{W}_{i}\delta_{\theta,A}'\tilde{X}_{i}\tilde{X}_{i}'\delta_{\beta,B}\right]=o_{P}(\sqrt{n})$.
	
	By\textbf{ }Lemmas \ref{lem: overall bnd theta} and \ref{lem: lasso bnd beta},
	$\delta_{\theta,A}\in\mathcal{C}({\rm supp}(\theta_{(1)}),3)$ and
	$\delta_{\beta,B}\in\mathcal{C}({\rm supp}(\beta_{(1)}),3)$ with
	probability tending to one. Let $S={\rm supp}(\theta_{(1)})\bigcup{\rm supp}(\beta_{(1)})$.
	Notice that $\mathcal{C}({\rm supp}(\theta_{(1)}),1)\subset\mathcal{C}(S,1)\subset\mathcal{C}(S,3)$
	and $\mathcal{C}({\rm supp}(\beta_{(1)}),3)\subset\mathcal{C}(S,3)$.
	It follows that $\PP(\delta_{\theta,A},\delta_{\beta,B}\in\mathcal{C}(S,3))\rightarrow1$.
	Hence, by $|S|\leq s_{\theta}+s_{\beta}$ and Lemma \ref{lem: matrix restricted eigen bernstein},
	we obtain 
	\begin{align*}
	& \left|\sum_{i=1}^{b_{n}}\left[\dot{q}(X_{i}'\theta_{(1)})W_{i}\delta_{\theta,A}'X_{i}X_{i}'\delta_{\beta,B}-\dot{q}(\tilde{X}_{i}'\theta_{(1)})\tilde{W}_{i}\delta_{\theta,A}'\tilde{X}_{i}\tilde{X}_{i}'\delta_{\beta,B}\right]\right|\\
	& =\left|\delta_{\theta,A}'\left(\sum_{i=1}^{b_{n}}\left[\dot{q}(X_{i}'\theta_{(1)})W_{i}X_{i}X_{i}'-\dot{q}(\tilde{X}_{i}'\theta_{(1)})\tilde{W}_{i}\tilde{X}_{i}\tilde{X}_{i}'\right]\right)\delta_{\beta,B}\right|\\
	&\qquad  \leq O_{P}\left(O_{P}\left(\sqrt{b_{n}\left(s_{\beta}+s_{\theta}\right)}\right)\right)\|\delta_{\theta,A}\|_{2}\|\delta_{\beta,B}\|_{2}
	\\
	& \qquad \overset{{\rm (i)}}{=}o_{P}\left(\sqrt{n\left(s_{\beta}+s_{\theta}\right)s_{\beta}s_{\theta}n^{-2}}\log p\right)\overset{{\rm (ii)}}{=}o_{P}(\sqrt{n}),
	\end{align*}
	where (i) follows by Lemmas \ref{lem: overall bnd theta} and \ref{lem: lasso bnd beta},
	together with $b_{n}\asymp n$ and (iii) follows by the conditions
	$s_{\theta}\ll\sqrt{n}/\log p$ and $s_{\beta}\lesssim n^{3/4}/\log p$. 
	
	\textbf{Step 4:} show that $\sum_{i=1}^{b_{n}}\left[\dot{q}(X_{i}'\hat{\theta}_{(1),A})-\dot{q}(X_{i}'\theta_{(1)})\right]W_{i}\delta_{\theta,A}'X_{i}X_{i}'\delta_{\beta,B}=o_{P}(\sqrt{n})$.
	
	Similar to Step 2, we can show that $\max_{1\leq i\leq b_{n}}|\dot{q}(X_{i}'\hat{\theta}_{(1),A})-\dot{q}(X_{i}'\theta_{(1)})|/|X_{i}'\delta_{\theta,A}|\leq O_{P}(1)$.
	Hence, 
	\begin{align*}
	& \left|\sum_{i=1}^{b_{n}}\left[\dot{q}(X_{i}'\hat{\theta}_{(1),A})-\dot{q}(X_{i}'\theta_{(1)})\right]W_{i}\delta_{\theta,A}'X_{i}X_{i}'\delta_{\beta,B}\right|\\
	& \leq O_{P}(1)\sum_{i=1}^{b_{n}}(X_{i}'\delta_{\theta,A})^{2}\cdot|X_{i}'\delta_{\beta,B}|\\
	& \leq O_{P}(1)\sqrt{\left(\sum_{i=1}^{b_{n}}(X_{i}'\delta_{\theta,A})^{4}\right)\left(\sum_{i=1}^{b_{n}}(X_{i}'\delta_{\beta,B})^{2}\right)}\\
	& \overset{{\rm (i)}}{\leq}O_{P}(1)\sqrt{O_{P}\left(b_{n}\|\delta_{\theta,A}\|_{2}^{4}\right)O_{P}\left(b_{n}\|\delta_{\beta,B}\|_{2}^{2}\right)}\\
	& \overset{{\rm (ii)}}{\leq}O_{P}(1)\sqrt{O_{P}\left([s_{\theta}b_{n}^{-1}\log p]^{2}\right)O_{P}\left(s_{\beta}\log p\right)}\overset{{\rm (iii)}}{=}o_{P}(\sqrt{n}),
	\end{align*}
	where (i) follows by Lemma \ref{lem: RE cond higer order}, together
	with $\|\theta_{(1)}\|_{0}\ll\sqrt{n}/\log p$ and the fact that $\PP(\delta_{\theta,A}\in\mathcal{C}({\rm supp}(\theta_{(1)}),1))\rightarrow1$
	(Lemma \ref{lem: overall bnd theta}), (ii) follows by Lemmas \ref{lem: overall bnd theta}
	and \ref{lem: lasso bnd beta} and (iii) follows by $b_{n}\asymp n$,
	$s_{\theta}\ll\sqrt{n}/\log p$ and $s_{\beta}\ll n/\log p$. The
	desired result follows by the above four steps together with (\ref{eq: debias sparse theta 1}).
\end{proof}
\begin{lem}
	\label{lem: debias sparse theta new}Under the assumptions of Theorem
	\ref{thm: sparse theta}, we have 
	\[
	\sum_{i\in H_{A}}W_{i}\left[q(X_{i}'\theta_{(1)})-q(X_{i}'\hat{\theta}_{(1),A})\right]X_{i}'\left(\hat{\beta}_{(1),B}-\beta_{(1)}\right)=o_{P}(\sqrt{n}).
	\]
\end{lem}
\begin{proof}[Proof of Lemma \ref{lem: debias sparse theta new}]
	Let $\delta_{\theta,A}=\hat{\theta}_{(1),A}-\theta_{(1)}$ and $\delta_{\beta,B}=\hat{\beta}_{(1),B}-\hat{\beta}_{(1)}$.
	Denote 
	$$\ddot{q}(a)=d^{2}q(a)/da^{2}=\exp(-a).$$
	 Define $Q_{i}=q(X_{i}'\theta_{(1)})-q(X_{i}'\hat{\theta}_{(1),A})+\dot{q}(X_{i}'\hat{\theta}_{(1),A})X_{i}'\delta_{\theta,A}$.
	Then 
	\begin{multline}
	\sum_{i\in H_{A}}W_{i}\left[q(X_{i}'\theta_{(1)})-q(X_{i}'\hat{\theta}_{(1),A})\right]X_{i}'\left(\hat{\beta}_{(1),B}-\beta_{(1)}\right)\\
	=\sum_{i\in H_{A}}W_{i}Q_{i}X_{i}'\delta_{\beta,B}-\sum_{i\in H_{A}}X_{i}'\delta_{\theta,A}W_{i}\dot{q}(X_{i}'\hat{\theta}_{(1),A})X_{i}'\delta_{\beta,B}.\label{eq: debias sparse theta new 1}
	\end{multline}
	
	By Lemma \ref{lem: debias sparse theta new part 1}, 
	\begin{equation}
	\sum_{i\in H_{A}}X_{i}'\delta_{\theta,A}W_{i}\dot{q}(X_{i}'\hat{\theta}_{(1),A})X_{i}'\delta_{\beta,B}=o_{P}(\sqrt{n}).\label{eq: debias sparse theta new 2}
	\end{equation}
	
	By assumption, $\|X\theta_{(1)}\|_{1}\leq\|X\|_{\infty}\|\theta_{(1)}\|_{1}\leq M_{1}M_{5}$.
	By Lemma \ref{lem: overall bnd theta}, $\PP(\|\hat{\theta}_{(1),A}\|_{1}\leq M_{5}\vee\kappa_{0})\rightarrow1$.
	Since $\|X\hat{\theta}_{(1),A}\|_{\infty}\leq\|X\|_{\infty}\|\hat{\theta}_{(1),A}\|_{1}\leq M_{1}\|\hat{\theta}_{(1),A}\|_{1}$,
	we have $\PP(\|X\hat{\theta}_{(1),A}\|_{\infty}\leq M_{1}(M_{5}\vee\kappa_{0}))\rightarrow1$.
	Therefore, by Taylor's theorem,
	\[
	\max_{1\leq i\leq n}\frac{|Q_{i}|}{(X_{i}'\delta_{\theta,A})^{2}}\leq\frac{1}{2}\sup_{|t|\leq\max\{\|X\theta_{(1)}\|_{\infty},\|X\hat{\theta}_{(1),A}\|_{\infty}\}}\ddot{q}(t)=O_{P}(1).
	\]
	
	Therefore,
	\begin{align}
	\left|\sum_{i\in H_{A}}W_{i}Q_{i}X_{i}'\delta_{\beta,B}\right| & \leq\sqrt{\left(\sum_{i\in H_{A}}Q_{i}^{2}\right)\left(\sum_{i\in H_{A}}(X_{i}'\delta_{\beta,B})^{2}\right)}\nonumber \\
	& \leq\sqrt{O_{P}(1)\left(\sum_{i\in H_{A}}(X_{i}'\delta_{\theta,A})^{4}\right)\left(\sum_{i\in H_{A}}(X_{i}'\delta_{\beta,B})^{2}\right)}\nonumber \\
	& \overset{{\rm (i)}}{\leq}\sqrt{O_{P}(1)b_{n}\left(s_{\theta}b_{n}^{-1}\log p\right)^{2}\left(s_{\beta}\log p\right)}\overset{{\rm (ii)}}{=}o_{P}(\sqrt{n}),\label{eq: debias sparse theta new 3}
	\end{align}
	where (i) follows by Lemmas \ref{lem: overall bnd theta} and \ref{lem: lasso bnd beta}
	and (ii) follows by $s_{\theta}\ll\sqrt{n}/\log p$, $s_{\beta}\ll n/\log p$
	and $b_{n}\asymp n$. 
	
	In light of (\ref{eq: debias sparse theta new 1}), the desired result
	follows by (\ref{eq: debias sparse theta new 2}) and (\ref{eq: debias sparse theta new 3}). 
\end{proof}
\subsubsection{Proof of Theorem \ref{thm: sparse theta}}
\begin{proof}[\textbf{Proof of Theorem \ref{thm: sparse theta}}]
	Similar to the proof of Theorem \ref{thm: sparse beta}, we observe
	the following decomposition
	\begin{equation}
	2b_{n}(\hat{\mu}_{(1)}-\mu_{(1)})=Q_{A}+Q_{B},\label{eq: thm sparse theta 1}
	\end{equation}
	where 
	\[
	Q_{A}=\sum_{i\in H_{A}}\left[W_{i}Y_{i}q(X_{i}'\hat{\theta}_{(1),A})+\left(1-W_{i}q(X_{i}'\hat{\theta}_{(1),A})\right)X_{i}'\hat{\beta}_{(1),B}-\mu_{(1)}\right]
	\]
	and 
	\[
	Q_{B}=\sum_{i\in H_{B}}\left[W_{i}Y_{i}q(X_{i}'\hat{\theta}_{(1),B})+\left(1-W_{i}q(X_{i}'\hat{\theta}_{(1),B})\right)X_{i}'\hat{\beta}_{(1),A}-\mu_{(1)}\right].
	\]
	
	We further decompose $Q_{A}$; an analogous argument applies for $Q_{B}$.
	Since $W_{i}Y_{i}=W_{i}Y_{i}(1)=W_{i}(X_{i}'\beta_{(1)}+\varepsilon_{i,(1)})$,
	we have 
	\[
	Q_{A}=\sum_{i\in H_{A}}\left[\left(1-W_{i}q(X_{i}'\theta_{(1)})\right)X_{i}'\beta_{(1)}+Y_{i}W_{i}q(X_{i}'\theta_{(1)})-\mu_{(1)}\right]+D_{n,1}+D_{n,2}+D_{n,3},
	\]
	where 
	\[
	\begin{cases}
	D_{1,n}=\sum_{i\in H_{A}}\left[1-W_{i}q(X_{i}'\theta_{(1)})\right]X_{i}'\left(\hat{\beta}_{(1),B}-\beta_{(1)}\right)\\
	D_{2,n}=\sum_{i\in H_{A}}W_{i}\varepsilon_{i,(1)}\left(q(X_{i}'\hat{\theta}_{(1),A})-q(X_{i}'\theta_{(1)})\right).\\
	D_{3,n}=\sum_{i\in H_{A}}W_{i}\left[q(X_{i}'\theta_{(1)})-q(X_{i}'\hat{\theta}_{(1),A})\right]X_{i}'\left(\hat{\beta}_{(1),B}-\beta_{(1)}\right)
	\end{cases}
	\]
	
	By the same argument as (\ref{eq: sparse beta eq 2}) in the proof
	of Theorem \ref{thm: sparse beta}, we can show 
	\[
	D_{2,n}=o_{P}(\sqrt{n}).
	\]
	
	Notice that $\hat{\beta}_{(1),B}$ is computed using observations
	in $H_{B}$ and is thus independent of $\{(W_{i},X_{i})\}_{i\in H_{A}}$.
	Also notice that $1-W_{i}q(X_{i}'\theta_{(1)})=-v_{i,(1)}q(X_{i}'\theta_{(1)})$
	and that conditional on $\{X_{i}\}_{i\in H_{A}}$, $\{v_{i,(1)}\}_{i\in H_{A}}$
	has mean zero and is independent across $i$. Thus, we have 
	\begin{align*}
	\EE\left(D_{1,n}^{2}\mid\{X_{i}\}_{i\in H_{A}},\hat{\beta}_{(1),B}\right) & =\sum_{i\in H_{A}}\EE\left(v_{i,(1)}^{2}\mid\{X_{i}\}_{i\in H_{A}}\right)\left(q(X_{i}'\theta_{(1)})\right)^{2}\left[X_{i}'\left(\hat{\beta}_{(1),B}-\beta_{(1)}\right)\right]^{2}\\
	& \overset{{\rm (i)}}{\leq}4\sum_{i\in H_{A}}\left(q(X_{i}'\theta_{(1)})\right)^{2}\left[X_{i}'\left(\hat{\beta}_{(1),B}-\beta_{(1)}\right)\right]^{2}\\
	& \leq4\left[\max_{i\in H_{A}}\left(q(X_{i}'\theta_{(1)})\right)^{2}\right]\sum_{i\in H_{A}}\left[X_{i}'\left(\hat{\beta}_{(1),B}-\beta_{(1)}\right)\right]^{2}\overset{{\rm (ii)}}{=}o_{P}(n),
	\end{align*}
	where (i) follows by $|v_{i,(1)}|=|W_{i}-e_{(1)}(X_{i})|\leq|W_{i}|+|e_{(1)}(X_{i})|\leq2$
	and (ii) follows by $\max_{1\leq i\leq n}q^{2}(X_{i}'\theta_{(1)})=O(1)$
	(due to the assumption of $\|X\theta_{(1)}\|_{\infty}=O(1)$) and
	Lemma \ref{lem: lasso bnd beta}. Hence, 
	\[
	D_{1,n}=o_{P}(\sqrt{n}).
	\]
	
	Lemma \ref{lem: debias sparse theta new} implies 
	\[
	D_{3,n}=o_{P}(\sqrt{n}).
	\]
	
	Thus, we have proved 
	\[
	Q_{A}=\sum_{i\in H_{A}}\left[\left(1-W_{i}q(X_{i}'\theta_{(1)})\right)X_{i}'\beta_{(1)}+Y_{i}W_{i}q(X_{i}'\theta_{(1)})-\mu_{(1)}\right]+o_{P}(\sqrt{n}).
	\]
	
	By an analogous argument, we can show that 
	\[
	Q_{B}=\sum_{i\in H_{B}}\left[\left(1-W_{i}q(X_{i}'\theta_{(1)})\right)X_{i}'\beta_{(1)}+Y_{i}W_{i}q(X_{i}'\theta_{(1)})-\mu_{(1)}\right]+o_{P}(\sqrt{n}).
	\]
	
	Since $2b_{n}/n\rightarrow1$, the desired result follows by (\ref{eq: thm sparse theta 1})
	and 
	$$\left(1-W_{i}q(X_{i}'\theta_{(1)})\right)X_{i}'\beta_{(1)}+Y_{i}W_{i}q(X_{i}'\theta_{(1)})-\mu_{(1)}=\varepsilon_{i,(1)}W_{i}q(X_{i}'\theta_{(1)})+X_{i}'\beta_{(1)}-\mu_{(1)}.$$ 
\end{proof}

\subsection{Proof of Lemma \ref{lem: asym var}}\label{sec: proof lem asy var}
\begin{proof}[\textbf{Proof of Lemma \ref{lem: asym var}}]
	Rearranging terms, we have 
	\begin{align*}
	V_{*} & =\EE \left[\left\{ W_{i}\varepsilon_{i,(1)}q(X_{i}'\theta_{(1)})+\left(X_{i}'\beta_{(1)}-\mu_{(1)}\right)\right\} -\left\{ (1-W_{i})\varepsilon_{i,(0)}q(X_{i}'\theta_{(0)})+\left(X_{i}'\beta_{(0)}-\mu_{(0)}\right)\right\} \right]^{2}\\
	& =\EE (\psi_{i,1}+\psi_{i,2})^{2},
	\end{align*}
	where $\psi_{i,1}=W_{i}\varepsilon_{i,(1)}q(X_{i}'\theta_{(1)})-(1-W_{i})\varepsilon_{i,(0)}q(X_{i}'\theta_{(0)})$
	and $\psi_{i,2}=X_{i}'(\beta_{(1)}-\beta_{(0)})-\tau$. Notice that
	$\EE(\psi_{i,1}\mid X_{i},W_{i})=0$. Therefore, 
	\[
	V_{*}=\EE \psi_{i,1}^{2}+\EE \psi_{i,2}^{2}.
	\]
	
	Since $W_{i}(1-W_{i})=0$, $\psi_{i,1}^{2}=W_{i}\varepsilon_{i,(1)}^{2}(q(X_{i}'\theta_{(1)}))^{2}+(1-W_{i})\varepsilon_{i,(0)}^{2}(q(X_{i}'\theta_{(0)}))^{2}$.
	The proof is complete. 
\end{proof}

\subsection{Proof of Theorem \ref{thm: consistency estim var}} \label{sec: proof consistency var estimate}
\begin{proof}[\textbf{Proof of Theorem \ref{thm: consistency estim var}}]
	By Lemma \ref{lem: asym var}, it suffices to show the following
	claims:
	\begin{enumerate}
		\item $\hat{V}_{1}=\tilde{V}_{1}+o_{P}(1)$, where $\tilde{V}_{1}=n^{-1}\sum_{i=1}^{n}W_{i}\varepsilon_{i,(1)}^{2}(q(X_{i}'\theta_{(1)}))^{2}$. 
		\item $\hat{V}_{2}=\tilde{V}_{2}+o_{P}(1)$, where $\tilde{V}_{2}=n^{-1}\sum_{i=1}^{n}(1-W_{i})\varepsilon_{i,(0)}^{2}(q(X_{i}'\theta_{(0)}))^{2}$. 
		\item $\hat{V}_{3}=\tilde{V}_{3}+o_{P}(1)$, where $\tilde{V}_{3}=n^{-1}\sum_{i=1}^{n}(X_{i}'(\beta_{(1)}-\beta_{(0)})-\tau)^{2}$. 
	\end{enumerate}
	This is because the law of large numbers would imply $\tilde{V}_{1}+\tilde{V}_{2}+\tilde{V}_{3}=V_{*}+o_{P}(1)$.
	We show the above three claims in three steps. 
	
	\textbf{Step 1:} show $\hat{V}_{1}=\tilde{V}_{1}+o_{P}(1)$.
	
	Notice that $W_{i}\varepsilon_{i,(1)}^{2}=W_{i}(Y_{i}-X_{i}'\beta_{(1)})^{2}$.
	Let $\delta_{\beta}=\hat{\beta}_{(1)}-\beta_{(1)}$ and $\delta_{\theta}=\hat{\theta}_{(1)}-\theta_{(1)}$.
	Then we have 
	\begin{multline}
	\hat{V}_{1}-\tilde{V}_{1}=\underset{T_{1}}{\underbrace{n^{-1}\sum_{i=1}^{n}W_{i}\left[(Y_{i}-X_{i}'\hat{\beta}_{(1)})^{2}-(Y_{i}-X_{i}'\beta_{(1)})^{2}\right](q(X_{i}'\hat{\theta}_{(1)}))^{2}}}\\
	+\underset{T_{2}}{\underbrace{n^{-1}\sum_{i=1}^{n}W_{i}\varepsilon_{i,(1)}^{2}\left[(q(X_{i}'\hat{\theta}_{(1)}))^{2}-(q(X_{i}'\theta_{(1)}))^{2}\right]}}.\label{eq: est var 2}
	\end{multline}
	
	We now bound these two terms. Notice that we have $\|\delta_{\beta,1}\|_{1}=O_{P}(\|\beta_{(1)}\|_{0}\sqrt{n^{-1}\log p})$,
	$\|\delta_{\theta,1}\|_{1}=O_{P}(\|\theta_{(1)}\|_{0}\sqrt{n^{-1}\log p})$
	and $\|\hat{\theta}_{(1)}\|_{1}=O_{P}(1)$. This is because these
	bounds hold for $\hat{\theta}_{(1),A}$ and $\hat{\theta}_{(1),B}$
	(Lemma \ref{lem: overall bnd theta}) as well as for $\hat{\beta}_{(1),A}$
	and $\hat{\beta}_{(1),B}$ (Lemma \ref{lem: lasso bnd beta}). We
	observe that
	\begin{align*}
	|T_{1}| & \leq\left(\max_{1\leq i\leq n}(q(X_{i}'\hat{\theta}_{(1)}))^{2}\right)n^{-1}\sum_{i=1}^{n}\left|(Y_{i}-X_{i}'\hat{\beta}_{(1)})^{2}-(Y_{i}-X_{i}'\beta_{(1)})^{2}\right|\\
	& \leq\left(\max_{1\leq i\leq n}(q(X_{i}'\hat{\theta}_{(1)}))^{2}\right)\left(n^{-1}\sum_{i=1}^{n}(X_{i}'\delta_{\beta,1})^{2}+\left|2n^{-1}\sum_{i=1}^{n}\varepsilon_{i,(1)}X_{i}'\delta_{\beta,1}\right|\right)\\
	& \leq\left(\max_{1\leq i\leq n}(q(X_{i}'\hat{\theta}_{(1)}))^{2}\right)\left(n^{-1}\sum_{i=1}^{n}(X_{i}'\delta_{\beta,1})^{2}+2\left\Vert n^{-1}\sum_{i=1}^{n}\varepsilon_{i,(1)}X_{i}\right\Vert _{\infty}\|\delta_{\beta,1}\|_{1}\right)\\
	& \overset{\text{(i)}}{=}O_{P}(1)\left(n^{-1}\sum_{i=1}^{n}(X_{i}'\delta_{\beta,1})^{2}+2\left\Vert n^{-1}\sum_{i=1}^{n}\varepsilon_{i,(1)}X_{i}\right\Vert _{\infty}\|\delta_{\beta,1}\|_{1}\right),
	\end{align*}
	where (i) follows by $\|X\hat{\theta}_{(1)}\|_{\infty}\leq\|X\|_{\infty}\|\hat{\theta}_{(1)}\|_{1}=O_{P}(1)$.
	Following an argument similar to the proof of the third claim in Lemma
	\ref{lem: RE condition}, we can easily show that $\left\Vert n^{-1}\sum_{i=1}^{n}\varepsilon_{i,(1)}X_{i}\right\Vert _{\infty}=O_{P}(\sqrt{n^{-1}\log p})$;
	essentially, the argument is Hoeffding inequality and the union bound
	since elements of $X_{i}\varepsilon_{i,(1)}$ have bounded sub-Gaussian
	norms. Therefore, the above display implies 
	\begin{align}
	|T_{1}| & =O_{P}(1)\left(n^{-1}\sum_{i=1}^{n}(X_{i}'\delta_{\beta,1})^{2}+2O_{P}(\sqrt{n^{-1}\log p})\|\delta_{\beta,1}\|_{1}\right)\nonumber \\
	& =O_{P}(1)\left(n^{-1}\sum_{i=1}^{n}(X_{i}'\delta_{\beta,1})^{2}+2O_{P}(\sqrt{n^{-1}\log p})\times O_{P}(\|\beta_{(1)}\|_{0}\sqrt{n^{-1}\log p})\right)\nonumber \\
	& \overset{\text{(i)}}{=}O_{P}\left(n^{-1}\sum_{i=1}^{n}(X_{i}'\delta_{\beta,1})^{2}\right)+o_{P}(1),\label{eq: est var 3}
	\end{align}
	where (i) holds by $\|\beta_{(1)}\|_{0}\ll n/\log p$. Let $\delta_{\beta,1,A}=\hat{\beta}_{(1),A}-\beta_{(1)}$
	and $\delta_{\beta,1,B}=\hat{\beta}_{(1),B}-\hat{\beta}_{(1)}$. Notice
	that 
	$$\delta_{\beta,1}=(\delta_{\beta,1,A}+\delta_{\beta,1,B})/2.$$
	By Lemma \ref{lem: lasso bnd beta}, $\PP(\delta_{\beta,1,A},\delta_{\beta,1,B}\in\mathcal{C}({\rm supp}(\beta_{(1)}),3))\rightarrow1$.
	By essentially the same argument as in the proof of the second claim
	of Lemma \ref{lem: RE condition}, we can show that 
	\[
	\max_{v\in\mathcal{C}({\rm supp}(\beta_{(1)}),3))}\frac{n^{-1}\sum_{i=1}^{n}(X_{i}'v)^{2}}{\|v\|_{2}^{2}}=O_{P}(1).
	\]
	
	By Lemma \ref{lem: lasso bnd beta} and $\|\beta_{(1)}\|_{0}\ll n/\log p$,
	we have $\|\delta_{\beta,1,A}\|_{2}=o_{P}(1)$ and $ $$\|\delta_{\beta,1,B}\|_{2}=o_{P}(1)$.
	It follows that $n^{-1}\sum_{i=1}^{n}(X_{i}'\delta_{\beta,1,A})^{2}=O_{P}(\|\delta_{\beta,1,A}\|_{2}^{2})=o_{P}(1)$
	and $n^{-1}\sum_{i=1}^{n}(X_{i}'\delta_{\beta,1,B})^{2}=O_{P}(\|\delta_{\beta,1,B}\|_{2}^{2})=o_{P}(1)$.
	Hence, the elementary bound yields 
	\begin{multline*}
	n^{-1}\sum_{i=1}^{n}(X_{i}'\delta_{\beta,1})^{2}=0.25n^{-1}\sum_{i=1}^{n}\left(X_{i}'\delta_{\beta,1,A}+X_{i}'\delta_{\beta,1,B}\right)^{2}\\
	\leq0.5n^{-1}\sum_{i=1}^{n}(X_{i}'\delta_{\beta,1,A})^{2}+0.5n^{-1}\sum_{i=1}^{n}(X_{i}'\delta_{\beta,1,B})^{2}=o_{P}(1).
	\end{multline*}

	In light of (\ref{eq: est var 3}), we have 
	\begin{equation}
	T_{1}=o_{P}(1).\label{eq: est var 4}
	\end{equation}
	
	Now we bound $T_{2}$. Let $f(x)=(q(x))^{2}$. Then 
	$$\dot{f}(x)=df(x)/dx=-2\exp(-2x)-2\exp(-x).$$
	By Taylor's theorem, there exists $r_{i}\in[0,1]$ such that 
	\[
	(q(X_{i}'\hat{\theta}_{(1)}))^{2}-(q(X_{i}'\theta_{(1)}))^{2}=(X_{i}'\delta_{\theta,1})\dot{f}\left(r_{i}X_{i}'\hat{\theta}_{(1)}+(1-r_{i})X_{i}'\theta_{(1)}\right).
	\]
	
	Since $\|\hat{\theta}_{(1)}\|_{1}=O_{P}(1)$, $\|\theta_{(1)}\|_{1}=O(1)$
	and $\|X\|_{\infty}=O(1)$, we have that 
	$$\max_{1\leq i\leq n}|r_{i}X_{i}'\hat{\theta}_{(1)}+(1-r_{i})X_{i}'\theta_{(1)}|=O_{P}(1).$$
	Therefore, 
	\[
	\max_{1\leq i\leq n}\left|\frac{(q(X_{i}'\hat{\theta}_{(1)}))^{2}-(q(X_{i}'\theta_{(1)}))^{2}}{X_{i}'\delta_{\theta,1}}\right|=O_{P}(1).
	\]
	
	Let $\delta_{\theta,1,A}=\hat{\theta}_{(1),A}-\theta_{(1)}$ and $\delta_{\theta,1,B}=\hat{\theta}_{(1),B}-\theta_{(1)}$.
	Thus, 
	\begin{align*}
	|T_{2}| & \leq n^{-1}\sum_{i=1}^{n}W_{i}\varepsilon_{i,(1)}^{2}\left|(q(X_{i}'\hat{\theta}_{(1)}))^{2}-(q(X_{i}'\theta_{(1)}))^{2}\right|\\
	& \leq O_{P}(1)n^{-1}\sum_{i=1}^{n}W_{i}\varepsilon_{i,(1)}^{2}|X_{i}'\delta_{\theta,1}|\\
	& \leq O_{P}(1)\sqrt{\left(n^{-1}\sum_{i=1}^{n}\varepsilon_{i,(1)}^{4}\right)\times\left(n^{-1}\sum_{i=1}^{n}(X_{i}'\delta_{\theta,1})^{2}\right)}\\
	& \overset{\text{(i)}}{=}O_{P}(1)\sqrt{O_{P}(1)\times\left(n^{-1}\sum_{i=1}^{n}(X_{i}'\delta_{\theta,1})^{2}\right)}\\
	& =O_{P}(1)\sqrt{O_{P}(1)\times\left(0.25n^{-1}\sum_{i=1}^{n}(X_{i}'\delta_{\theta,1,A}+X_{i}'\delta_{\theta,1,B})^{2}\right)}\\
	& \leq O_{P}(1)\sqrt{O_{P}(1)\times\left(0.5n^{-1}\sum_{i=1}^{n}(X_{i}'\delta_{\theta,1,A})^{2}+0.5n^{-1}\sum_{i=1}^{n}(X_{i}'\delta_{\theta,1,B})^{2}\right)},
	\end{align*}
	where (i) follows by the law of large numbers and the fact that $\varepsilon_{i,(1)}$
	is sub-Gaussian. Now by essentially the same argument as above, we
	can show that 
	$$n^{-1}\sum_{i=1}^{n}(X_{i}'\delta_{\theta,1,A})^{2}=O_{P}(\|\delta_{\theta,1,A}\|_{2}^{2})$$
	and similarly $n^{-1}\sum_{i=1}^{n}(X_{i}'\delta_{\theta,1,B})^{2}=O_{P}(\|\delta_{\theta,1,B}\|_{2}^{2})$.
	By Lemma \ref{lem: overall bnd theta} and the condition of $\|\theta_{(1)}\|_{0}\ll n/\log p$,
	we have that 
	$$\|\delta_{\theta,1,A}\|_{2}=o_{P}(1)$$ and $\|\delta_{\theta,1,B}\|_{2}=o_{P}(1)$.
	Therefore, the above display implies $T_{2}=o_{P}(1)$. Thus, in light
	of (\ref{eq: est var 2}) and (\ref{eq: est var 4}), we have proved
	$\hat{V}_{1}-\tilde{V}_{1}=o_{P}(1)$. 
	
	\textbf{Step 2:} show $\hat{V}_{2}=\tilde{V}_{2}+o_{P}(1)$.
	
	The argument is completely analogous to Step 1 and is thus omitted. 
	
	\textbf{Step 3:} show $\hat{V}_{3}=\tilde{V}_{3}+o_{P}(1)$.
	
	Let $\xi_{i}=X_{i}'(\beta_{(1)}-\beta_{(0)})-\tau$, $\hat{\xi}_{i}=X_{i}'(\hat{\beta}_{(1)}-\hat{\beta}_{(0)})-\hat{\tau}$
	and $u_{i}=\hat{\xi}_{i}-\xi_{i}$. Notice that 
	\begin{align}
	\left|\hat{V}_{3}-\tilde{V}_{3}\right|&=\left|n^{-1}\sum_{i=1}^{n}\left((\xi_{i}+u_{i})^{2}-\xi_{i}^{2}\right)\right| 
	\nonumber \\
	& \leq\left|n^{-1}\sum_{i=1}^{n}u_{i}^{2}\right|+2\left|n^{-1}\sum_{i=1}^{n}\xi_{i}u_{i}\right|\nonumber \\
	& \leq\left|n^{-1}\sum_{i=1}^{n}u_{i}^{2}\right|+2\sqrt{\left(n^{-1}\sum_{i=1}^{n}\xi_{i}^{2}\right)\times\left(n^{-1}\sum_{i=1}^{n}u_{i}^{2}\right)}\nonumber \\
	& =\left|n^{-1}\sum_{i=1}^{n}u_{i}^{2}\right|+\sqrt{O_{P}(1)\times\left(n^{-1}\sum_{i=1}^{n}u_{i}^{2}\right)}.\label{eq: est var 6}
	\end{align}
	
	Thus, it only remains to show $n^{-1}\sum_{i=1}^{n}u_{i}^{2}=o_{P}(1)$.
	Let $\delta_{\beta,j}=\hat{\beta}_{(j)}-\beta_{(j)}$ for $j\in\{0,1\}$.
	By elementary inequality of $(a+b+c)^{2}\leq4a^{2}+4b^{2}+2c^{2}$,
	we have 
	\[
	u_{i}^{2}=\left(X_{i}'\delta_{\beta,1}-X_{i}'\delta_{\beta,0}-(\hat{\tau}-\tau)\right)^{2}\leq4(X_{i}'\delta_{\beta,1})^{2}+4(X_{i}'\delta_{\beta,0})^{2}+2(\hat{\tau}-\tau)^{2}.
	\]
	
	Thus, 
	\[
	n^{-1}\sum_{i=1}^{n}u_{i}^{2}\leq4n^{-1}\sum_{i=1}^{n}(X_{i}'\delta_{\beta,1})^{2}+4n^{-1}\sum_{i=1}^{n}(X_{i}'\delta_{\beta,0})^{2}+2(\hat{\tau}-\tau)^{2}.
	\]
	
	By the arguments in Step 1 and Step 2, we have already proved $n^{-1}\sum_{i=1}^{n}(X_{i}'\delta_{\beta,j})^{2}=o_{P}(1)$
	for $j\in\{0,1\}$. Since $\sqrt{n}(\hat{\tau}-\tau)\rightarrow^{d}N(0,V_{*})$,
	$\hat{\tau}-\tau=o_{P}(1)$. Therefore, $n^{-1}\sum_{i=1}^{n}u_{i}^{2}=o_{P}(1)$.
	It follows by (\ref{eq: est var 6}) that $\hat{V}_{3}=\tilde{V}_{3}+o_{P}(1)$.
	The proof is therefore complete. 
\end{proof}

\bibliographystyle{plainnat}
\bibliography{references}

\end{document}